\documentclass[11pt]{article}
\usepackage{arxiv}

\usepackage[utf8]{inputenc} 
\usepackage[T1]{fontenc}    
\usepackage{hyperref}       
\usepackage{xcolor}
\hypersetup{
    colorlinks=true,
    linkcolor=blue,
    filecolor=magenta,      
    urlcolor=cyan,
    }
\usepackage{url}            
\usepackage{booktabs}       
\usepackage{amsfonts}       
\usepackage{nicefrac}       
\usepackage{microtype}      
\usepackage{lipsum}		
\usepackage{graphicx}
\usepackage{doi}

\usepackage{amsmath}
\usepackage{bm}

\usepackage{enumerate}
\usepackage{enumitem}

\usepackage{amsthm}
\newtheorem{theorem}{Theorem}[section]

\usepackage{nicematrix}

\theoremstyle{remark}
\newtheorem*{remark}{Remark}

\title{An Integro-differential Model of Cadmium Yellow Photodegradation}

\author{ \href{https://orcid.org/0000-0002-7055-9323}{\includegraphics[scale=0.1]{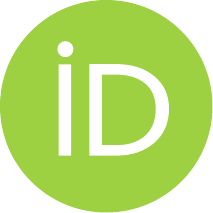}}\hspace{1mm}Maurizio~Ceseri\thanks{Deceased on August $5^{\text{th}}$, 2024.} \\
	C.N.R. National Research Council of Italy,\\
	Institute for Applied Mathematics {\small``Mauro Picone''}\\
	Via dei Taurini, 19 - 00185 Rome, Italy. \\
	\texttt{maurizio.ceseri@cnr.it} \\
	\And
	\href{https://orcid.org/0000-0002-5453-6383}{\includegraphics[scale=0.1]{orcid.pdf}}\hspace{1mm}Roberto~Natalini\thanks{Members of the INdAM Research group GNAMPA} \\
	C.N.R. National Research Council of Italy,\\
	Institute for Applied Mathematics {\small``Mauro Picone''}\\
	Via dei Taurini, 19 - 00185 Rome, Italy. \\
	\texttt{roberto.natalini@cnr.it} \\
	\And
	\href{https://orcid.org/0000-0002-1869-945X}{\includegraphics[scale=0.1]{orcid.pdf}}\hspace{1mm}Mario~Pezzella\thanks{Members of the INdAM Research group GNCS} \\
 	C.N.R. National Research Council of Italy,\\
	Institute for Applied Mathematics {\small``Mauro Picone''}\\
	Via P. Castellino, 111 - 80131 Naples - Italy. \\
	\texttt{mario.pezzella@cnr.it} \\
}

\hypersetup{
pdftitle={An Integro-differential Model of Cadmium Yellow Photodegradation},
pdfsubject={math.DS, math.NA},
pdfauthor={Maurizio~Ceseri, Roberto~Natalini, Mario~Pezzella},
pdfkeywords={Integro-differential models, photochemical reactions, cultural heritage, positivity-preserving numerical methods},
}

\begin{document}
\maketitle
\begin{abstract}
	Many paintings from the $19^{\text{th}}$ century have exhibited signs of fading and discoloration, often linked to cadmium yellow, a pigment widely used by artists during that time. In this work, we develop a mathematical model of the cadmium sulfide photocatalytic reaction responsible for these damages. By employing non-local integral operators, we capture the interplay between chemical processes and environmental factors, offering a detailed representation of the degradation mechanisms.
    Furthermore, we present a second order positivity-preserving numerical method designed to accurately simulate the phenomenon and ensure reliable predictions across different scenarios, along with a comprehensive sensitivity analysis of the model. 
\end{abstract}


\section{Introduction}
Cadmium yellow \cite{Original_Yellow_Pigments,faulkner2009high,fiedler1986cadmium}, whose synthesis was originally described by Gay-Lussac in 1818, was extensively employed by artists throughout the $19^{\text{th}}$ and $20^{\text{th}}$ centuries, including Pablo Picasso \cite{Picasso2,Picasso1}, Joan Miró \cite{Miro1}, Edvard Munch \cite{Urlo_Munch_Degrad,Monico_Urlo}, Henri Matisse \cite{Matisse2,Matisse3,Pouyet2015}, Claude Monet \cite{Monet1}, James Ensor \cite{Van_der_Snickt} and Vincent van Gogh \cite{VanderSnickt2012}.
When exposed to light, this synthetic pigment, primarily composed of cadmium sulfide ($CdS$), undergoes a photocatalytic reaction that results in color degradation, posing a significant challenge to the long-term preservation of pictorial matrices and cultural heritage. The first systematic study of the $CdS$ deterioration mechanism dates back to 2005 with the work \cite{Leone_First}, which investigated paintings from 1887 to 1923 using X-ray diffraction and scanning electron microscopy. Since then, a range of significant multidisciplinary contributions has emerged in the scientific literature, furthering the knowledge on the topic \cite{  Computational_CdS1,Computational_CdS2,Computational_CdS3,monico2018role,Pisu,Pisu_Materials}.  

From a chemical perspective, $CdS$ reacts with environmental humidity and oxygen to form cadmium sulfate ($CdSO_4$), with light acting as an activator. Specifically, since the cadmium sulfide behaves as a semiconductor with a definite band-gap energy \cite{Band_Gap} $E_{bg}=2.42$ eV, the reaction is initiated by supra-band-gap light with photon energy 
\begin{equation*}
    E_M=\dfrac{\texttt{h}\texttt{c}}{\lambda_M}> E_{bg}
\end{equation*}
and therefore by light within a spectral range corresponding to wavelengths shorter than $\lambda_M=512.331$ nm (here,  $c=2.99792\cdot 10^{17}$ $\text{nm} \cdot \text{s}^{-1}$ denotes the speed of light and $h=4.13567 \cdot 10^{-15}$ $ \text{s}\cdot\text{eV}$ is the Planck constant). Moreover, observations indicate that temperature has a minimal impact on the deterioration effect, which is more pronounced with increased relative humidity. The overall process, described in \cite{monico2018role} as follows
\begin{equation}\label{eq:chemical}
CdS + 4h^+_{surf} + 2H_2O + O_2 \longrightarrow CdSO_4 + 4H^+,
\end{equation}
results in a noticeable color change and the formation of a thin ($5-8$ \textmu m) layer of cadmium sulfate on the paint surface. 

The objective of this work is twofold. Firstly, due to the absence of similar approaches in the literature, we develop a novel integro-differential mathematical model that represents the photochemical degradation process of cadmium yellow, with particular emphasis on the non-local effects of light exposure. Secondly, to ensure efficient and realistic simulations of the phenomenon, we design an accurate numerical method that unconditionally preserves the essential properties of the continuous model, namely the positivity and monotonicity of its solutions. Therefore, we aim to construct a comprehensive mathematical framework that enhances the understanding of the phenomenon and supports the development of effective preservation strategies for cultural heritage.

The manuscript is structured as follows. Section \ref{sec:M_Model} presents the derivation of our model and establishes the foundation for the subsequent analysis.  In Section \ref{sec:Num_Method} we introduce a dynamically consistent numerical method and provide evidence of its quadratic convergence and advantageous performances. Some simulations under various scenarios are discussed in Section \ref{sec:Simulations} and a detailed local sensitivity analysis of the model is conducted in Section \ref{sec:Sensitivity}. Finally, closing remarks and perspectives for future investigation in Section \ref{sec:Conclusions} conclude the paper.

\section{Mathematical Modeling}\label{sec:M_Model} We develop an integro-differential mathematical model to describe the chemical degradation of a thin cadmium yellow paint with thickness $L$, exposed to a humidity level $w(z)$ and illumination $I_0(\lambda)$ that are constant over time. In what follows, $z \in [0, L]$ represents the depth within the paint and $\lambda \in [\lambda_m, \lambda_M]$ is the light wavelength. Let $c(z, t)$ and $g(z, t)$ denote the concentrations, at time $t \in [0, \mathsf{T}],$ of $CdS$ and $CdSO_4$, respectively. The reaction in \eqref{eq:chemical} is modeled by the following fully conservative Production-Destruction System (PDS)
\begin{equation}\label{eq:PDS_Chem}
    \begin{cases}
        \dfrac{\partial c}{\partial t}(z,t)=-k_1(c(z,t),g(z,t)) \ c(z,t) \ w(z), \\[0.25cm]
        \dfrac{\partial g}{\partial t}(z,t)=k_1(c(z,t),g(z,t))\ c(z,t)\ w(z), \qquad\qquad (z,t)\in [0,L]\times [0,\mathsf{T}],
    \end{cases}
\end{equation}
where $c_0(z)=c(z,0),$ $g_0(z)=g(z,0)$ are known initial conditions and $k_1(\cdot)$ is the concentrations-dependent photodegradation rate,  which will be defined in detail later. The conservativity property of the PDS \eqref{eq:PDS_Chem} (we refer to \cite{anderHeiden1982} for further details, or to \cite{PDS} and references therein) ensures that, regardless of the form of $k_1(\cdot)$, the following linear invariant holds
\begin{equation*}
    c(z,t) + g(z,t) = c_0(z) + g_0(z), \qquad \forall \ (z,t) \in [0,L] \times [0,\mathsf{T}].
\end{equation*}
Therefore, under the assumption that no cadmium sulfate is present at the onset of the process, i.e. $g_0(z)=0$ for all $0\leq z \leq L,$ we get $g(z,t)=c_0(z)-c(z,t)$ and the equation for the $CdSO_4$ can be disregarded.

To accurately model the reaction \eqref{eq:chemical}, the kinetic rate $ k_1(\cdot)$ incorporates several dependencies. According to the Beer-Lambert law \cite{d2012simplified}, the light intensity is defined as follows
\begin{equation*}
    I\left(\lambda,\int_0^z    c(\zeta,t)  \ d\zeta\right)=I_0(\lambda)\exp\left\{-\int_0^z\mu_a(\lambda,c(\zeta,t),g(\zeta,t)) \ d\zeta\right\}, \qquad \lambda\in[\lambda_m,\lambda_M],
\end{equation*}
where the system's absorbance $\mu_a$ reads
\begin{equation*}
    \mu_a(\lambda,c(z,t),g(z,t))=\varepsilon_c(\lambda)c(z,t)+\varepsilon_g(\lambda)g(z,t)=(\varepsilon_c(\lambda)-\varepsilon_g(\lambda))c(z,t)+\varepsilon_g(\lambda)c_0(z),
\end{equation*}
with $\varepsilon_c(\lambda)$ and $\varepsilon_g(\lambda)$ representing the molar absorptivities of $CdS$ and $CdSO_4$, respectively (their relation to reflectance is discussed in Section \ref{sec:Simulations}). Here, following the approach in \cite{clarelli2013fluid,Eilers,Thébault}, we describe the overall light penetration effect with the term
\begin{equation*}
    \int_{\lambda_m}^{\lambda_M}\frac{2\hat{I}\left(\lambda,\int_0^z    c(\zeta,t)  \ d\zeta\right)}{\hat{I}^2\left(\lambda,\int_0^z    c(\zeta,t)  \ d\zeta\right)+1} \ d\lambda,
\end{equation*}
where $\hat{I}\left(\cdot\right)=I\left(\cdot\right)/I_{ref}$ and  $I_{ref}$ is a given reference value. Incorporating the effects of light, of the environmental temperature $T$ via the Arrhenius law \cite{Atkins_2016} and of the humidity level $w(z),$ into the PDS \eqref{eq:PDS_Chem} yields
\begin{equation}\label{eq:5}
    \begin{cases}
         \displaystyle\dfrac{\partial c}{\partial t}(z,t)=-A\exp\left\{-\frac{E_a}{RT}\right\}w(z) \ c(z,t)\int_{\lambda_m}^{\lambda_M}\dfrac{2\hat{I}(\lambda, \int_0^z    c(\zeta,t)  \ d\zeta)}{\hat{I}^2(\lambda, \int_0^z    c(\zeta,t)  \ d\zeta)+1} \ d\lambda, \\[0.4cm]
         \displaystyle\hat{I} \! \left( \! \lambda, \! \int_0^z  \!\! c(\zeta,t) \ d\zeta \! \right) \! = \! \dfrac{I_0(\lambda)}{I_{ref}} \exp  \left\{ \! -  \mu \! \left( \! (\varepsilon_c(\lambda) -\varepsilon_g(\lambda)) \! \int_0^z \!\!   c(\zeta,t)  \ d\zeta + \varepsilon_g(\lambda) C_0(z) \! \right) \! \right\} \! ,
    \end{cases}
\end{equation}
where $z\in [0,L]$, $t\in [0,\mathsf{T}]$, $C_0(z)=\int_0^z   c_0(\zeta)  \ d\zeta,$ $E_a$ is the reaction activation energy and $R$ is the gas constant. 

\begin{remark}
    As experimental observations have not yet detected any processes of spatial diffusion, the system in \eqref{eq:PDS_Chem} does not explicitly account for variations along the variable $z \in [0,L]$. However, the spatial dependence of the phenomenon is incorporated in the model \eqref{eq:5} through the use of non-linear, non-local integral operators.
\end{remark}

\subsection{Dimensionless Formulation}\label{subsec:Adim_Model} In this subsection we reformulate the mathematical model \eqref{eq:5} in a non-dimensional form. Assuming that the reference values 
    \begin{equation*}
         [c^{ref}]=\text{mol}/\text{cm}^3, \quad [w^{ref}]=\text{mol}/\text{cm}^3, \quad [\varepsilon_c^{ref}]=\text{cm}^{2}/\text{mol},\quad [\varepsilon_g^{ref}]=\text{cm}^{2}/\text{mol},
    \end{equation*}
are given, we introduce the dimensionless parameters 
    \begin{equation}\label{eqref:parametri_ADIM}
       \nu = \dfrac{\varepsilon_c^{ref}}{\varepsilon_g^{ref}}, \qquad \mu=L c^{ref} \varepsilon_g^{ref}, \qquad  \xi=(\lambda_M-\lambda_m) w^{ref}  \mathsf{T}A\exp\left\{-\frac{E_a}{RT}\right\}.
    \end{equation}
Furthermore, denoted by $\Omega=[0,1]$, we linearly rescale the variables as follows 
\begin{equation}\label{eq:transformations}
   z\in [0,L]\mapsto \tilde{z}=\dfrac{z}{L}, \quad  t \in [0,\mathsf{T}] \mapsto \tilde{t}=\dfrac{t}{\mathsf{T}}, \quad \lambda  \in  [\lambda_m,\lambda_M] \mapsto \tilde{\lambda}=\dfrac{\lambda-\lambda_m}{\lambda_M-\lambda_m}, 
\end{equation}
and define, for $\tilde{z},$ $\tilde{t}$ and $\tilde{\lambda}$ belonging to $\Omega,$ the dimensionless functions
    \begin{equation}\label{eqref:funzioni_ADIM}
        \begin{split}
            &\tilde{c}(\tilde{z},\tilde{t})=\dfrac{c(\tilde{z} L,\tilde{t} \mathsf{T})}{c^{ref}}, \qquad \;\; \tilde{c}_0(\tilde{z})=\dfrac{c_0(\tilde{z} L)}{c^{ref}}, \qquad \;\; \tilde{w}(\tilde{z})=\dfrac{w(\tilde{z}  L)}{w^{ref}}, \\ 
            &\tilde{I}(\tilde{\lambda})=\frac{I_0(\lambda_0+\tilde{\lambda} (\lambda_M-\lambda_m))}{I_{ref}}, \qquad \quad  \   
            \tilde{\varepsilon}_g(\tilde{\lambda})=\dfrac{\varepsilon_g(\lambda_0+\tilde{\lambda} (\lambda_M-\lambda_m) )}{\varepsilon_g^{ref}}, \\
            &\tilde{\varepsilon}_c(\tilde{\lambda})=\dfrac{\varepsilon_c(\lambda_0+\tilde{\lambda} (\lambda_M-\lambda_m))}{\varepsilon_c^{ref}}, 
            \qquad \quad 
            \tilde{\varepsilon}_\nu(\tilde{\lambda})=\nu \tilde{\varepsilon}_c(\tilde{\lambda}) - \tilde{\varepsilon}_g(\tilde{\lambda}). \\
        \end{split}
    \end{equation}
From now on, the initial $CdS$ concentration is assumed to be constant in space, so that $c_0(z)=c^{ref}$ for all $z \in [0, L],$ and $\tilde{c}_0(\tilde{z})=1$ for all $\tilde{z} \in \Omega$. Moreover, we exclusively refer to the non-dimensional variables and functions, which, for simplicity, will still be denoted without the tilde.

The model \eqref{eq:5} is then restated as follows
\begin{equation}\label{eq:Continuous_model:ADIM}
    \frac{\partial c}{\partial t}(z,t)=-\xi \ w(z)c(z,t)   \int_{0}^{1}\dfrac{2\hat{I}\left(\lambda,z,\int_0^z c(\zeta,t) \ d\zeta\right)}{\hat{I}^2\left(\lambda,z,\int_0^z c(\zeta,t) \ d\zeta\right)+1}\ d\lambda, \qquad  (z,t)\in \Omega \times \Omega,
\end{equation}
with 
\begin{equation}\label{eq:Adimensional_I}
    \hat{I}\left( \lambda,z,\int_0^z   c(\zeta,t) \ d\zeta \right)=I(\lambda) \exp\left\{ -  \mu \left(  \varepsilon_\nu (\lambda)   \int_0^z    c(\zeta,t)  \ d\zeta +z \varepsilon_g(\lambda) \right) \right\}.
\end{equation}
Compared to the original formulation \eqref{eq:5}, the dimensionless form \eqref{eq:Continuous_model:ADIM}-\eqref{eq:Adimensional_I} offers several advantages, including the reduction in the number of parameters.

\section{Unconditionally Positive and Monotonic Numerical Method}\label{sec:Num_Method}
Since the function $c(z,t)$ represents the dimensionless concentration of a chemical reactant undergoing degradation, our focus is on solutions to \eqref{eq:Continuous_model:ADIM} that are positive and monotonically non-increasing over time. Remarkably, the development of numerical methods that are both of high order and unconditionally preserve the aforementioned properties, poses significant challenges. Non-standard finite difference discretizations have proven effective in constructing positivity-preserving schemes for integro-differential systems (cf. \cite{MPV_NSFD,MPV_NLFD} and references therein). However, the linear convergence of the resulting integrators, which may represent a limitation in simulating realistic scenarios, motivates us to adopt a different approach built on a suitable reformulation of the continuous model. 

Straightforward manipulations of \eqref{eq:Continuous_model:ADIM} lead to the following implicit Volterra Integral Equation (VIE) (see, for instance, \cite[Section 3.4]{Brunner_2004}) 
\begin{equation}\label{eq:VFIE_model}
        \displaystyle c(z,t)=\exp\left\{-\xi \ w(z) \int_{0}^t  \int_{0}^{1}\dfrac{2\hat{I}\left(\lambda,z,\int_0^z c(\zeta,\tau) \ d\zeta\right)}{\hat{I}^2\left(\lambda,z,\int_0^z c(\zeta,\tau) \ d\zeta\right)+1}\ d\lambda \ d \tau \right\}, \qquad  (z,t)\in \Omega \times \Omega, 
\end{equation}
with a non-local non-linearity depending on the function $\hat{I}(\cdot)$ defined in \eqref{eq:Adimensional_I}. The idea of the exponential reformulation, introduced in \cite{MPV_Axioms,MPV_2024} for epidemic models based on second kind convolution VIEs, has the advantage of yielding unconditionally positive schemes when the integrals are approximated. However, although the direct discretization of the embedded memory operators in \eqref{eq:VFIE_model} with quadrature rules might provide a viable option for designing accurate schemes, it would result in purely implicit methods and computationally demanding algorithms. Moreover, ensuring unconditional monotonicity of the numerical solution could be challenging in such cases. Here, we introduce a fully explicit, second order scheme which preserves both the positivity and the monotonicity of the solution with no limitations on the discretization steplengths.

Let $\Delta z,\Delta t,\Delta \lambda$ $>0$ represent the dimensionless spatial, temporal and frequency stepsizes, respectively. Define $N_z,$ $N_t$ and $N_\lambda$ as positive integers such that $1=N_z \Delta z=N_t \Delta t=N_\lambda \Delta \lambda.$ Consider the uniform meshes
$z_j=j\Delta z$ for $0\leq j\leq N_z,$ $t_n=n\Delta t$ for $0\leq n\leq N_t$ and $\lambda_l=l \Delta \lambda$ for $0\leq l\leq N_\lambda.$ We denote by $F$ the function $x\in \mathbb{R}^+_0\to (2x)/(x^2+1)\in \mathbb{R}^+_0$ and by $c^j_n\approx c(z_j,t_n)$ the approximation of the solution to \eqref{eq:VFIE_model}. In this setting, given $c_0^j=1,$ $j=0,\dots,N_z,$ we propose the fully-explicit numerical method
\begin{equation}\label{eq:PC method ADIM}
    \begin{cases}
        \alpha_{n}^{j,l}   = \displaystyle I(\lambda_l)\exp\left\{- \mu\Delta z \left( \varepsilon_\nu (\lambda_l)\sum_{r=0}^{j-1}c_n^r + j\varepsilon_g(\lambda_l)  \right)\right\}, \\
        p_{n+1}^j =c_n^j \displaystyle \exp\left\{-\Delta t \Delta\lambda \ \xi w(z_j)    \sum_{l=0}^{N_{\lambda}-1} F(\alpha_n^{j,l}) \right\}, \\
        \beta_{n}^{j,l}   = \displaystyle I(\lambda_l)\exp\left\{-\mu \dfrac{\Delta z}{2} \left(\varepsilon_{\nu}(\lambda_l)\left(c_n^0+2\sum_{r=1}^{j-1}c_n^r+c_n^j\right) + 2j \varepsilon_g(\lambda_l) \right)\right\},\\
        \beta_{n+1}^{j,l}   = \displaystyle I(\lambda_l)\exp\left\{-\mu \dfrac{\Delta z}{2} \left(\varepsilon_{\nu}(\lambda_l)\left(p_{n+1}^0+2\sum_{r=1}^{j-1}p_{n+1}^r+p_{n+1}^j\right) + 2j \varepsilon_g(\lambda_l)  \right)\right\}, \\
        \gamma^j_{n+1}    = \displaystyle F(\beta_n^{j,0})+F(\beta_{n+1}^{j,0})+ 2\!\sum_{l=1}^{N_{\lambda}-1}\!\left(F(\beta_n^{j,l}) + F(\beta_{n+1}^{j,l})\right)+ F(\beta_n^{j,N_{\lambda}})+F(\beta_{n+1}^{j,N_{\lambda}}), \\
         c_{n+1}^j   \displaystyle=c_n^j \exp\left\{-\dfrac{\Delta t}{2} \dfrac{\Delta \lambda}{2} \ \xi w(z_j)\gamma^j_{n+1}    \right\}, \qquad\qquad\qquad\qquad \;\;
         \begin{array}{l}
              n=0,\dots, N_t-1, \\
              j=0,\dots,N_z, 
         \end{array}\\
    \end{cases}
\end{equation}
formulated utilizing a Predictor-Corrector (PC) strategy (see, for instance, \cite{Garey1972}). More specifically, a rectangular rule discretization of the integrals in \eqref{eq:VFIE_model} is here adopted for the predictive step, while the correction process is designed mimicking the trapezoidal rule (we refer to \cite{Davis2014} and references therein for further details on quadrature discretizations).

The following result addresses the unconditional positivity and monotonicity of the PC numerical solution.
\begin{theorem}
    Consider equation \eqref{eq:PC method ADIM} and assume that $I(\lambda)\geq 0,$ for all $\lambda\in\Omega.$ Then, independently of the positive values $\Delta x$, $\Delta t$, $\Delta \lambda,$ the solution $\{c^j_n\}_{n \in \mathbb{N}_0},$ $j\geq 0$ is positive and non-increasing with respect to $n$.
\end{theorem}
\begin{proof}
    The positivity of the sequence $\{c^j_n\}_{n \in \mathbb{N}_0}$ naturally follows, independently of $j\geq 0,$ from the last equation in \eqref{eq:PC method ADIM}. By definition $\beta_{n}^{j,l}$ and $\beta_{n+1}^{j,l}$ are positive as well, so that  $F:\mathbb{R}^+_0\to \mathbb{R}^+_0$ implies $\gamma^j_{n+1}>0,$ for any $n$ and $j.$ Standard inductive arguments then establish the monotonicity property $c_{n+1}^j\leq c_{n}^j\leq 1$ for all $j$ and $n\in \mathbb{N}_0.$
\end{proof}

\subsection{Convergence and Performances}
The following result establishes the quadratic convergence of the numerical method \eqref{eq:PC method ADIM} as $\Delta t,\Delta z,\Delta \lambda \to 0^+$. The proof, which is based on a combination of the arguments presented in \cite[Section 3]{Garey1972} and \cite[Section 3.1]{MPV_2024}, is highly technical and is therefore omitted here, as it goes beyond the scope of this work. 

\begin{theorem}\label{thm:convergence}
    Assume that the known function $I(\lambda)$ is continuously differentiable on $\Omega.$ Denote by $c(z,t),$ for $(z,t)\in \Omega \times \Omega$, the continuous solution to \eqref{eq:Continuous_model:ADIM} and by $\{c^j_n, \; 0\leq n \leq N_t, \; 0\leq j \leq N_z\}$ its approximation computed by the PC scheme \eqref{eq:PC method ADIM}. Let $N_z,$ $N_t$ and $N_\lambda$ be positive integers such that $1=N_z \Delta z=N_t \Delta t=N_\lambda \Delta \lambda.$ Then, the global discretization error $E(\delta z, \Delta t, \Delta \lambda ; t_n,z_j)=|c(z_j,t_n)-c^j_n|$ satisfies
    \begin{equation*}
        \max_{\substack{0\leq j \leq N_z \\ 0\leq n \leq N_t}} E(\Delta z, \Delta t, \Delta \lambda ; t_n,z_j) \leq \Theta  (\Delta z^2+\Delta t^2+\Delta \lambda^2),
    \end{equation*}
    with $\Theta$ positive constant independent of the stepsizes. Therefore, the scheme \eqref{eq:PC method ADIM} is convergent of order two.
\end{theorem}
The theoretical aspects of Theorem \ref{thm:convergence}, as well as alternative approaches for the simulation of model \eqref{eq:5}, will be further investigated and clarified in \cite{Mario_Pezzella}.

In order to experimentally show the convergence property stated with Theorem \ref{thm:convergence}, we address equations \eqref{eq:Continuous_model:ADIM}-\eqref{eq:Adimensional_I} with
\begin{equation}\label{eq:toy_problem_convergence}
    \begin{array}{l}
         \xi=1, \quad \ w(z)=1-z, \\
         \mu=1, \quad \ I(\lambda)=\exp\{\lambda\}-1,\\
    \end{array} \;\; \varepsilon_\nu (\lambda)=\frac{\lambda-0.9}{1.2-\lambda}, \quad \varepsilon_g(\lambda)=1-\lambda+2\lambda^2-\frac{4}{5}\lambda^3,
\end{equation}
where the chosen functions and parameters do not possess any specific physical meaning. We integrate the problem \eqref{eq:toy_problem_convergence} with the PC scheme \eqref{eq:PC method ADIM} and compute the mean space-time error $E_{c}$ and the experimental order of convergence $\rho_c$ as follows
    \begin{equation}\label{eq:Error}
        E_{c}(\Delta z, \Delta t; \Delta \lambda)=\Delta z\Delta t\sum_{j=0}^{N_z}\sum_{n=0}^{N_t} \left|c_n^j-C_n^j\right|, \qquad \rho_c=\log_2\left(\dfrac{E_{c}(\Delta z, \Delta t; \Delta \lambda)}{E_{c}(\frac{\Delta z}{2}, \frac{\Delta t}{2}; \frac{\Delta \lambda}{2})}\right).
    \end{equation}
Here, the reference solution $C_n^j$ is computed by the same method with $\Delta z=\Delta t=\Delta \lambda=2^{-12}$. Figure \ref{fig:Exp_Convergence} and Table \ref{tab:Exp_Convergence} report the simulations outcomes for different values of the stepsizes and confirm the second order convergence of the PC discretization \eqref{eq:PC method ADIM}. Furthermore, in order to assess the efficiency of the correction process, we compare the PC numerical solution with the one obtained by 
\begin{equation}\label{eq:solo_predictor}
     \pi_{n+1}^j \! =\pi_n^j \displaystyle \exp \! \left\{ \! -\Delta t \Delta\lambda  \xi w(z_j) \! \!  \sum_{l=0}^{N_{\lambda}-1} \!\! \! F \! \left( \! I(\lambda_l)\exp \! \left\{ \! - \mu\Delta z \! \left( \varepsilon_\nu (\lambda_l)\sum_{r=0}^{j-1}\pi_n^r + j\varepsilon_g(\lambda_l)  \! \right) \! \right\} \! \right) \! \right\} \! ,
\end{equation}
which corresponds to the positive and monotone first order scheme used in \eqref{eq:PC method ADIM} as the Predictor (P). The work precision diagram in Figure \ref{fig:Exp_Convergence} shows the mean errors with respect to the computational efforts and confirms the superior performances of \eqref{eq:PC method ADIM} compared to \eqref{eq:solo_predictor}. Remarkably, the PC method achieves an accuracy at least two orders of magnitude higher within a fixed execution time. 

\begin{table}[htbp]
\small
\caption{Experimental convergence of the P and the PC methods.}\label{tab:Exp_Convergence}
\begin{center}
  \begin{tabular}{c|cc|cc} 
   \multicolumn{1}{c|}{\bf Stepsizes} & \multicolumn{2}{c|}{\bf Mean Space-time Errors} & \multicolumn{2}{c}{\bf Exp. order} \\ 
    $\Delta x=\Delta t=\Delta \lambda$ & \bf $E_\pi$ & \bf $E_c$ & \bf $\rho_{\pi}$ & \bf $\rho_c$ \\ 
    \hline \rule{-1.5pt}{9pt}
    $\!2^{-3\phantom{1}}$& $1.07\cdot 10^{-2}$  & $5.08\cdot 10^{-4}$ & $-$    & $-$    \\
    $2^{-4\phantom{1}}$  & $5.28\cdot 10^{-3}$  & $1.29\cdot 10^{-4}$ & $1.02$ & $1.98$ \\ 
    $2^{-5\phantom{1}}$  & $2.27\cdot 10^{-3}$  & $3.24\cdot 10^{-5}$ & $1.01$ & $1.99$ \\ 
    $2^{-6\phantom{1}}$  & $1.31\cdot 10^{-3}$  & $8.11\cdot 10^{-6}$ & $1.00$ & $2.00$ \\ 
    $2^{-7\phantom{1}}$  & $6.54\cdot 10^{-4}$  & $2.03\cdot 10^{-6}$ & $1.00$ & $2.00$ \\ 
    $2^{-8\phantom{1}}$  & $3.27\cdot 10^{-4}$  & $5.06\cdot 10^{-7}$ & $1.00$ & $2.00$ \\ 
    $2^{-9\phantom{1}}$  & $1.63\cdot 10^{-4}$  & $1.25\cdot 10^{-7}$ & $1.00$ & $2.02$ \\ 
    $2^{-10}$ & $8.16\cdot 10^{-5}$  & $2.98\cdot 10^{-8}$ & $1.00$ & $2.07$ \\ 
    $2^{-11}$ & $4.08\cdot 10^{-5}$  & $5.95\cdot 10^{-9}$ & $1.00$ & $2.32$ \\ 
    \hline
  \end{tabular}
\end{center}
\end{table}

\begin{figure}[htbp]
  \centering
  \label{fig:Exp_Convergence}\includegraphics[scale=0.70]{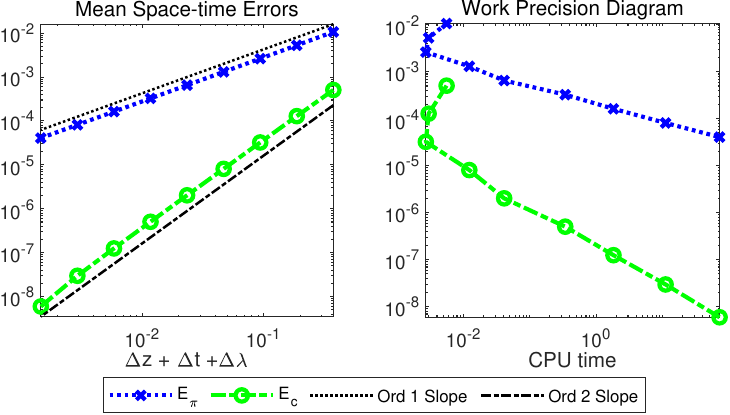}
  \caption{Experimental convergence and performances of the P and PC methods.}
\end{figure}

\section{Numerical Simulations}\label{sec:Simulations}
In this section we present several experiments under varying scenarios to provide insights into the qualitative behavior of the developed mathematical model. The dimensionless formulation \eqref{eq:Continuous_model:ADIM} is here simulated via the PC method \eqref{eq:PC method ADIM}, choosing the stepsizes $\Delta z=\Delta t=\Delta \lambda=5\cdot 10^{-4}$. The dimensional numerical solution is then obtained by reversing the transformations in \eqref{eq:transformations}. 

\vspace{0.5\baselineskip}
\textbf{Test 1: the BCT}\label{defn: BCT} - For our first experiment, hereafter referred to as the \textit{Base Case Test} (BCT), we assume that the light hits the paint surface at $z=0$ and set the parameters as detailed in Table \ref{tab:Parameters_Base_Case}.

The function $I(\lambda),$ shown in Figure \ref{fig:Irrad}, is derived by a local quadratic regression of the emission data for the UV-filtered xenon lamp reported in \cite[Fig. 1]{Monico_Dati_per_I}. The molar absorptivity $\varepsilon_c(\lambda)$ of cadmium sulfide is obtained from the hex-$CdS$ diffuse reflectance UV-Vis spectra in \cite[Supporting Information, Sec. 2.1]{monico2018role}. Specifically, given the reflectance $R(\lambda)$ at a relative humidity level of $95\%$ we assume a zero transmittance and, following the arguments in \cite{gobrecht2015combining}, we compute 
\begin{equation}\label{eq:e_c Values}
    \varepsilon_c(\lambda)=-\dfrac{\log\left(R(\lambda)\right)}{c^{ref} \cdot L_c}\dfrac{1}{\varepsilon_c^{ref}}, \qquad \qquad\lambda\in\Omega,
\end{equation}
where $L_c=10^{-4} \ \text{cm}$ represents the thickness of the cadmium sulfate top crust. As for the $CdSO_4$, in absence of experimental data, we assume a proportionality between the two molar absorptivities and take $\varepsilon_g(\lambda)=\nu^{-1}\varepsilon_c(\lambda),$ with $\nu$ defined in \eqref{eqref:parametri_ADIM} (cf. Figure \ref{fig:Eps} for the plot of the \hyperref[defn: BCT]{BCT}). Moreover, the infiltration of water is here characterized by the following depth-dependent profile
\begin{equation}\label{eq:acqua_livello}
    w(z)=\left(1-\dfrac{w_b}{w^{ref}}\right)(1-z), \qquad \qquad \qquad z\in\Omega,
\end{equation}
where $w_b$ denotes the lowest level of humidity below which the reaction does not occur. 

\begin{table}[htbp]
\small
\caption{Parameters of the \hyperref[defn: BCT]{BCT} simulation.}\label{tab:Parameters_Base_Case}
\begin{center}
    \begin{tabular}{c|c|c|c}
        \textbf{Params.} & \textbf{Descriptions} & \textbf{Units} & \textbf{Values} \\
        \hline \rule{-3pt}{9pt}
        $L$          &  Depth of painted layer      & cm & $7.00 \cdot 10^{-3\phantom{1}}$           \\
        $\mathsf{T}$ & Reference time       & s  & $2.30 \cdot 10^{6\phantom{-1}}$  \\
        $\lambda_m$  & Minimum wavelength   & nm & $3.80 \cdot 10^{2\phantom{-1}}$  \\
        $\lambda_M$  & Band-gap wavelength  & nm & $5.12 \cdot 10^{2\phantom{-1}} $ \\ 
        $A$  & Arrhenius pre-exponential factor & $\text{cm}^2 \ \text{mol}^{-1} \text{s}^{-1}$ & $1.00 \cdot 10^{8\phantom{-1}} $ \\
        $E_a$ & Activation energy & $\text{J} \ \text{mol}^{-1}$ & $7.78 \cdot 10^{-19} $ \\
         $R$  & Gas constant & $\text{J} \ \text{K}^{-1} \ \text{mol}^{-1}$  & $8.31 \cdot 10^{0\phantom{-1}} $ \\
        $T$  & Temperature & $\text{K}$ & $2.98 \cdot 10^{2\phantom{-1}} $ \\
        $c^{ref}$  & Reference $CdS$ concentration   & $\text{mol} \ \text{cm}^{-3}$ & $3.34 \cdot 10^{-2\phantom{1}}$  \\
        $w^{ref}$  & Humidity reference value   & $\text{mol} \ \text{cm}^{-3}$ & $1.22 \cdot 10^{-6\phantom{1}}$ \\
        $w_b$  & Humidity lower threshold   & $\text{mol} \ \text{cm}^{-3}$ & $5.77 \cdot 10^{-7\phantom{1}}$ \\
        $\varepsilon_c^{ref}$  &  $CdS$ molar absorptivity reference value   & $\text{cm}^{2} \ \text{mol}^{-1}$ & $1.64 \cdot 10^{5\phantom{-1}}$ \\
        $\varepsilon_g^{ref}$  & $CdSO_4$ molar absorptivity reference value &  $\text{cm}^{2} \ \text{mol}^{-1}$ & $6.56 \cdot 10^{5\phantom{-1}}$ \\
        $I^{ref}$  & Irradiance reference value &  $\text{W}^{2} \ \text{m}^{-2} \ \text{nm}^{-1}$ & $3.48 \cdot 10^{0 \phantom{-1}}$ \\ 
        \hline 
    \end{tabular} 
    \vspace{17pt}
    \begin{tabular}{cccc}
        \multicolumn{4}{l}{\bf Dimensionless Parameters $ \qquad \mu=1.53\cdot 10^2, \quad \nu=2.50\cdot 10^{-1}, \quad \ \xi=3.70\cdot 10^3$} 
    \end{tabular}
\end{center}
\end{table}
\begin{figure}[htbp]
  \centering
  \begin{minipage}{0.45\textwidth}
    \centering
    \includegraphics[width=.9\linewidth]{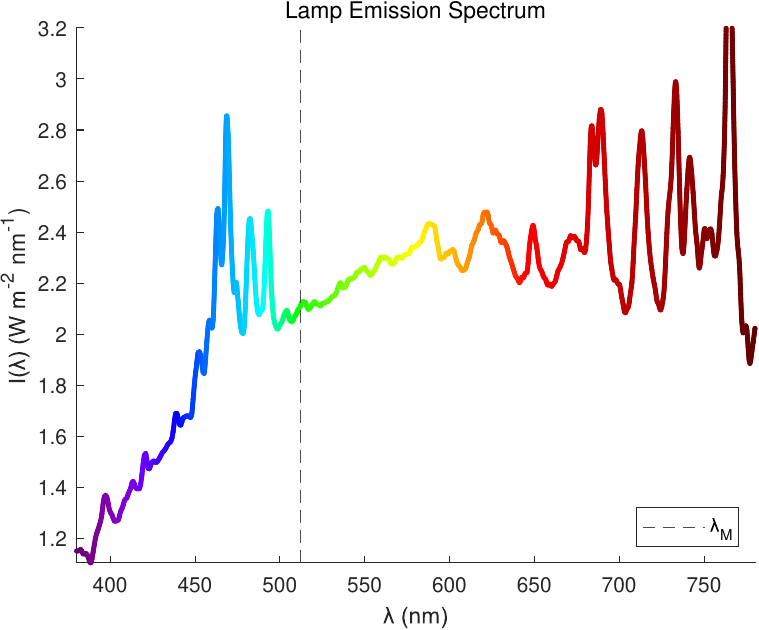}
    \caption{Absolute irradiance as a function of the wavelength for the UV-filtered xenon lamp.}
    \label{fig:Irrad}
  \end{minipage}\hspace{0.05\textwidth} 
  \begin{minipage}{0.45\textwidth}
    \centering
    \includegraphics[width=.9\linewidth]{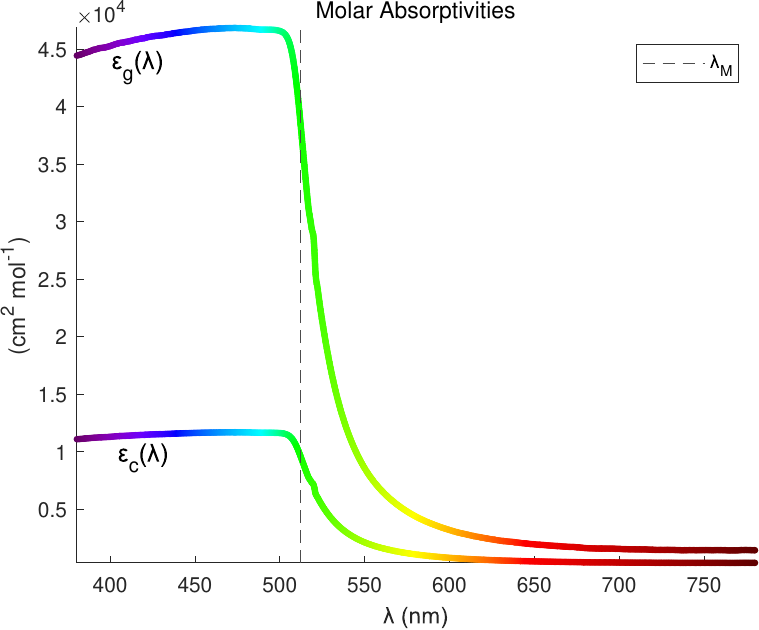}
    \caption{Molar absorptivities as functions of the wavelength for the \hyperref[defn: BCT]{BCT} $(\varepsilon_g(\lambda)=4\varepsilon_c(\lambda))$.}
    \label{fig:Eps}
  \end{minipage}
\end{figure}

The PC numerical simulations of the \hyperref[defn: BCT]{BCT} in Figures \ref{fig:BCT_2D} and \ref{fig:BCT} demonstrate that the photochemical degradation over time predominantly affects the surface, leading to significant variations in the concentrations of $CdS$ and $CdSO_4$ within the shallow regions (cf. Figures \ref{fig:BCT_2D_Zoom} and \ref{fig:BCT_Zoom}). The emergence of an S-shaped pattern in Figures \ref{fig:BCT} and \ref{fig:BCT_Zoom} highlights the formation of cadmium sulfate on the surface and its gradual penetration deeper into the layers. A passivation effect induced by the $CdSO_4$ is evident as well.  Indeed, higher concentrations of cadmium sulfate, which is white and highly reflective, visibly retard the overall reaction process. 

\begin{figure}[htbp]
  \centering
  \begin{minipage}{0.45\textwidth}
    \centering
    \includegraphics[width=.9\linewidth]{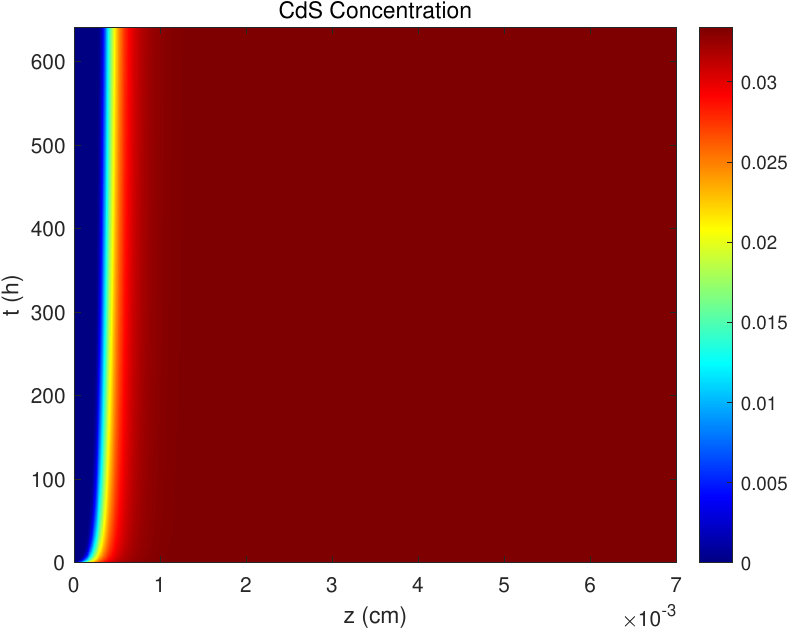}
    \caption{\hyperref[defn: BCT]{BCT}: Space-time evolution of the cadmium sulfide concentration}
    \label{fig:BCT_2D}
  \end{minipage} \hspace{0.05\textwidth} 
  \begin{minipage}{0.45\textwidth}
    \centering
    \includegraphics[width=.9\linewidth]{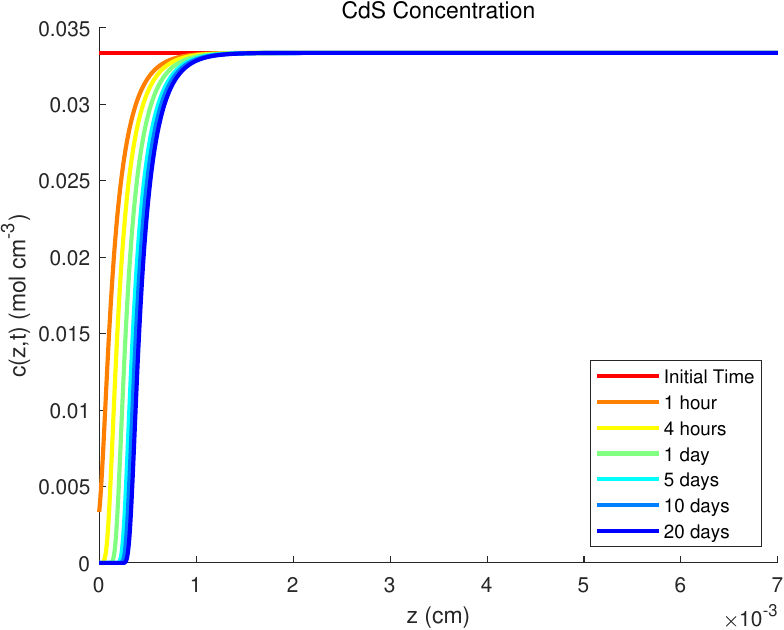}
    \caption{\hyperref[defn: BCT]{BCT}: Space evolution of the cadmium sulfide concentration at different times}
    \label{fig:BCT}
  \end{minipage}
\end{figure}
\begin{figure}[htbp]
  \centering
  \begin{minipage}{0.45\textwidth}
    \centering
    \includegraphics[width=.9\linewidth]{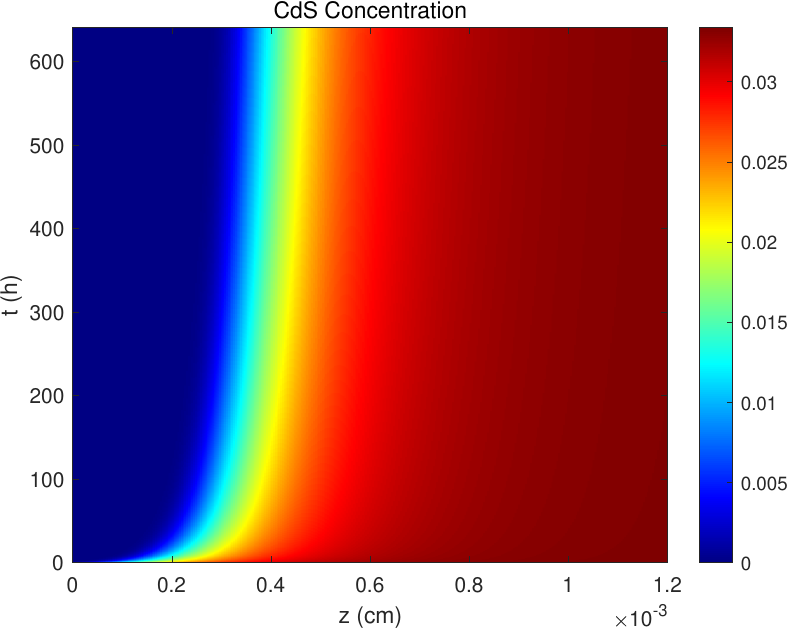}
    \caption{\hyperref[defn: BCT]{BCT}: Space-time evolution, near the surface, of the cadmium sulfide concentration.}
    \label{fig:BCT_2D_Zoom}
  \end{minipage} \hspace{0.05\textwidth} 
  \begin{minipage}{0.45\textwidth}
    \centering
    \includegraphics[width=.9\linewidth]{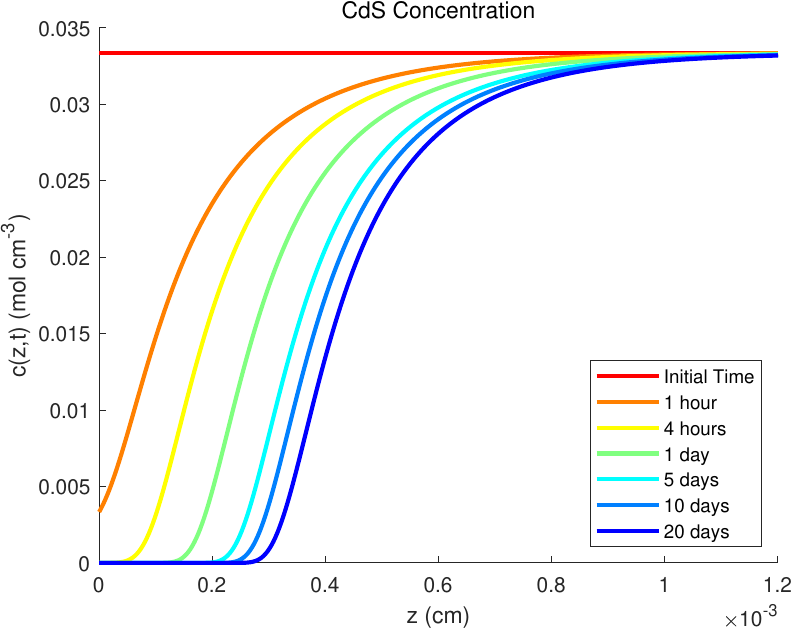}
    \caption{\hyperref[defn: BCT]{BCT}: Space evolution, near the surface, of the cadmium sulfide concentration for different times.}
    \label{fig:BCT_Zoom}
  \end{minipage}
\end{figure}

To quantify the passivation effect, given a threshold value $\psi \in \mathring{\Omega}$, we analyze the behavior of the dimensionless separation front 
\begin{equation}\label{eq:Separation_Front}
    \Sigma_\psi(t)=\min\{z\in \Omega \ : \ c(z,t)\geq \psi \}, \qquad t\in \Omega,
\end{equation}
that represents the spatial point from which, at time $t,$ the $CdSO_4$ concentration goes below the threshold $1-\psi$. Figure \ref{fig:BC_Sep_Front} illustrates the approximated values $\Sigma_\psi^n\approx\Sigma_\psi(t_n)$ derived from the PC numerical solution when $\psi\in\{0.7,0.8,0.9\}.$ These plots indicate that the separation front progresses linearly respectively to the logarithm of time, whatever the value of $\psi.$ Therefore, we identify $\Sigma_\psi(t)$ in \eqref{eq:Separation_Front} with the following function 
\begin{equation*}
    S_\psi(t)=a(\psi)+b(\psi)\log(t),
\end{equation*}
where the threshold-dependent coefficients are determined by a fitting procedure minimizing the mismatch with the $\Sigma_\psi^n$ values (see Table \ref{tab:Parameters_Log_Sig} and Figure \ref{fig:BC_Log_Sep_Front} for the outcomes). It then follows that $\frac{d^2}{dt^2}S_\psi(t) \! \propto \! -\frac{1}{t^2}$ represents an estimate, as time goes by, of the reaction speed reduction due to passivation. 

\begin{table}[htbp]
\footnotesize
\caption{Fitted parameters of the separation front logarithmic representation for the \hyperref[defn: BCT]{BCT} simulation.}\label{tab:Parameters_Log_Sig}
\begin{center}
    \begin{tabular}{c|c c|c}
        \textbf{Threshold} & \multicolumn{2}{c|}{\textbf{Fitted Parameters}} & \textbf{Mean Fitting} \\
        \textbf{Value} $\psi$ & $a(\psi)$ &  $b(\psi)$ & \textbf{Error on} $\{\Sigma^n_\psi\}$\\
        \hline \rule{-3pt}{9pt}
        $0.9$  &     $6.53 \cdot 10^{-3}$ & $9.69 \cdot 10^{-2}$ & $9.24 \cdot 10^{-7}$ \\
        $0.8$ & $6.57 \cdot 10^{-3}$ & $8.09 \cdot 10^{-2}$ & $9.21 \cdot 10^{-7}$  \\
        $0.7$ & $6.60 \cdot 10^{-3}$ & $7.20 \cdot 10^{-2}$ & $9.15 \cdot 10^{-7}$ \\
        \hline 
    \end{tabular} 
\end{center}
\end{table}

\begin{figure}[htbp]
  \centering
  \begin{minipage}{0.45\textwidth}
    \centering
    \includegraphics[width=.9\linewidth]{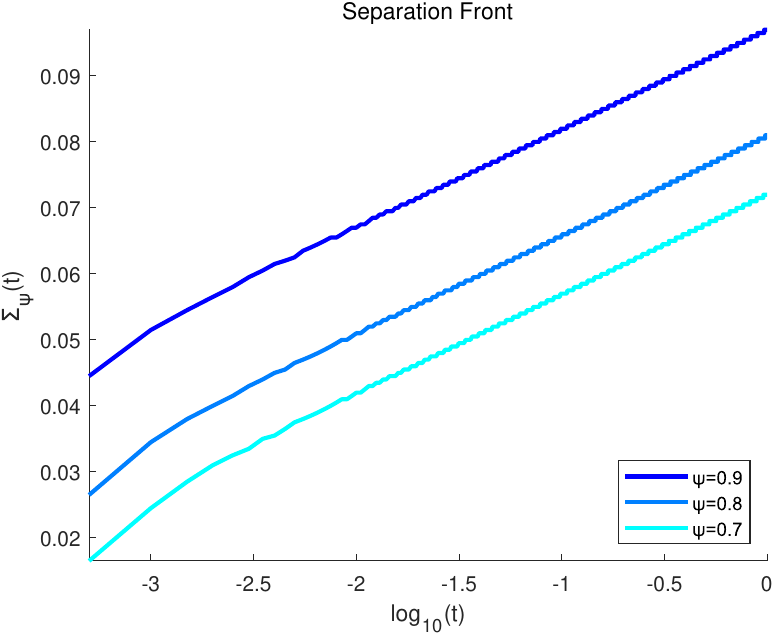}
    \caption{\hyperref[defn: BCT]{BCT}: Separation front evolution with respect to the logarithm of time.}
    \label{fig:BC_Sep_Front}
  \end{minipage} \hspace{0.05\textwidth} 
  \begin{minipage}{0.45\textwidth}
    \centering
    \includegraphics[width=.9\linewidth]{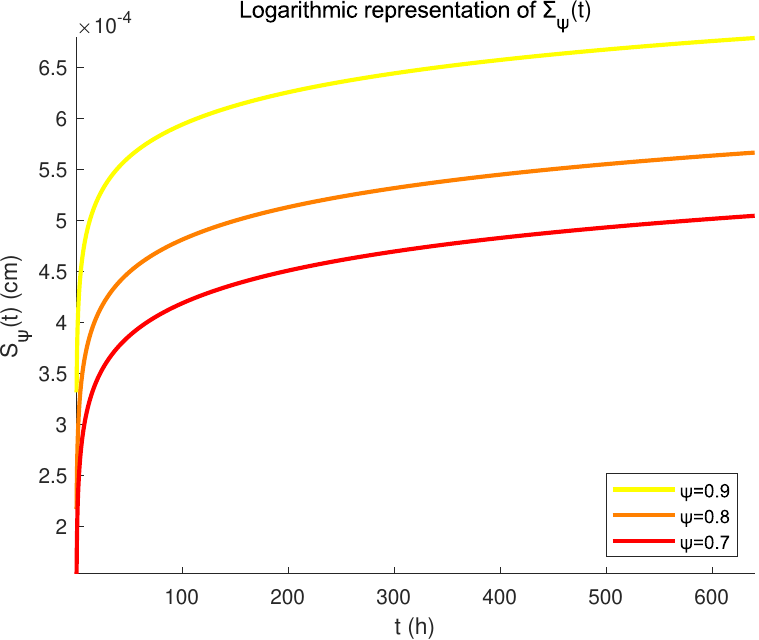}
    \caption{\hyperref[defn: BCT]{BCT}: Fitted logarithmic representation of the separation front}
    \label{fig:BC_Log_Sep_Front}
  \end{minipage}
\end{figure}

\vspace{0.5\baselineskip}
\textbf{Test 2: the APT}\label{defn: APT} - In all simulations, the function $\varepsilon_c(\lambda)$ is derived from data as specified in \eqref{eq:e_c Values}. However, due to the lack of experimental data for cadmium sulfate, we assume that $\varepsilon_g(\lambda) = \nu^{-1}\varepsilon_c(\lambda).$ To investigate how this assumption on molar absorptivities affects the degradation kinetics, we perform an \textit{Absorptivity Proportionality Test} (APT).

We explore various choices for the proportionality factor $\nu$, reported in Table \ref{tab:Parameters_Base_Case}, while maintaining the other parameters consistent with the \hyperref[defn: BCT]{BCT}. Figure \ref{fig:APT_Final} reports a comparison of the final $CdS$ concentrations corresponding to the different values of $\nu.$ The spatial variation, with respect to the \hyperref[defn: BCT]{BCT}, is reported in Figure \ref{fig:APT_Difference}. Numerical simulations clearly indicate that as $\varepsilon_g(\lambda)$ increases relatively to $\varepsilon_c(\lambda)$, the degradation effect on $CdS$ decreases. These findings align with expectations, as higher molar absorptivity of cadmium sulfate, which tends to accumulate on the surface, leads to greater attenuation of light penetration and its associated effects. 

\begin{table}[htbp]
\footnotesize
\caption{Parameters of the \hyperref[defn: APT]{APT} simulation.}\label{tab:Parameters_APT}
\begin{center}
    \begin{tabular}{c|c|c|c|c|c|c}
        \textbf{Parameters} & \multicolumn{5}{c|}{\textbf{Values}}  &  \textbf{Units} \\
        \hline \rule{-3pt}{9pt}
        $\varepsilon_c^{ref}$  &     $1.64 \cdot 10^{5}$ &     $1.64 \cdot 10^{5}$ &  $1.64 \cdot 10^{5}$ &  $1.64 \cdot 10^{5}$ &  $1.64 \cdot 10^{5}$ &  $\text{cm}^{2} \ \text{mol}^{-1}$ \\
        $\varepsilon_g^{ref}$ & $1.64 \cdot 10^{5}$ & $3.28 \cdot 10^{5}$ &  $6.56 \cdot 10^{5}$ &  $1.32 \cdot 10^{6}$ &  $2.64 \cdot 10^{6}$ &  $\text{cm}^{2} \ \text{mol}^{-1}$ \\
        $\nu$  & $1$ & $1/2$ &  $1/4$ &  $1/8$ &  $1/16$ &  $-$ \\
        \hline 
    \end{tabular} 
\end{center}
\end{table}

\begin{figure}[htbp]
  \centering
  \begin{minipage}{0.45\textwidth}
    \centering
    \includegraphics[width=.9\linewidth]{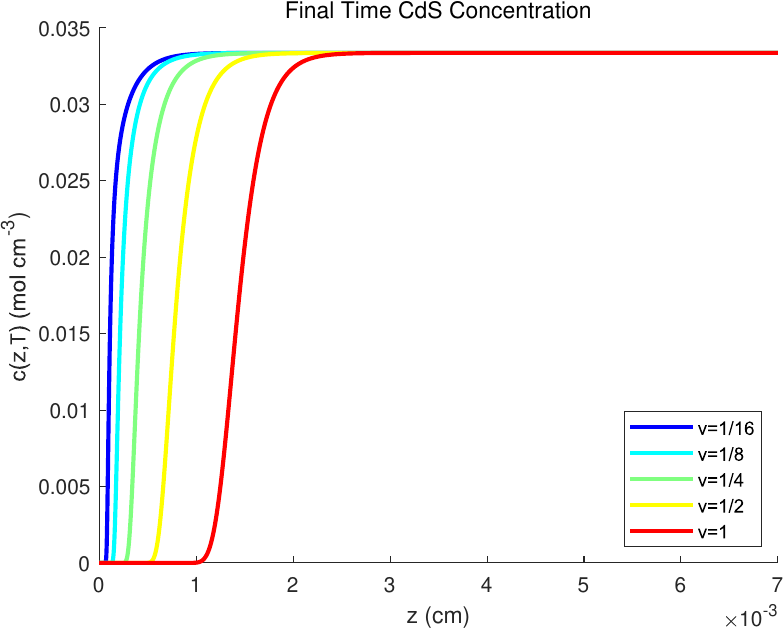}
    \caption{\hyperref[defn: APT]{APT}: Space evolution, after $20$ days, of the cadmium sulfide concentration for different values of $\nu$.}
    \label{fig:APT_Final}
  \end{minipage} \hspace{0.05\textwidth} 
  \begin{minipage}{0.45\textwidth}
    \centering
    \includegraphics[width=.9\linewidth]{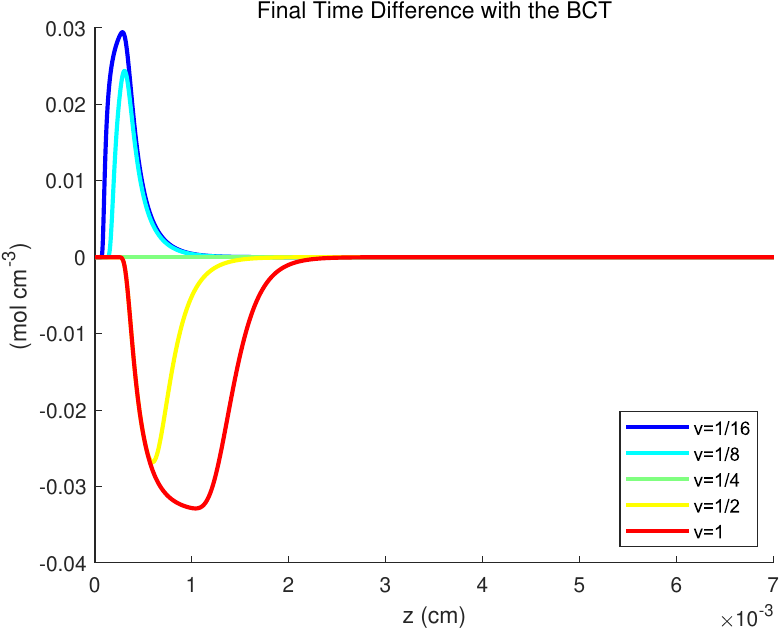}
    \caption{\hyperref[defn: APT]{APT}: Variation of the CdS final concentration with respect to the \hyperref[defn: BCT]{BCT}, \textit{i.e.} the case of $\nu=1/4.$}
    \label{fig:APT_Difference}
  \end{minipage}
\end{figure}

\vspace{0.5\baselineskip}
\textbf{Test 3: the UVT}\label{defn: UVT} - Until now, we have excluded ultraviolet (UV) radiation from our analysis based on consultations with practitioners, who noted that indoor lighting effectively filters out this component. However, given the potential to include paintings exposed outdoors, we conduct an experiment that incorporates these previously omitted wavelengths. Specifically, for our \textit{Ultraviolet Test} (UVT), we consider $\lambda_m=200 \ \text{nm}$ and employ the functions plotted in Figures \ref{fig:Irrad_UV} and \ref{fig:Eps_UV}, which are derived from the data in \cite{monico2018role,Monico_Dati_per_I}. The remaining parameters are equal to those defined in the \hyperref[defn: BCT]{BCT}.

\begin{figure}[htbp]
  \centering
  \begin{minipage}{0.45\textwidth}
    \centering
    \includegraphics[width=.9\linewidth]{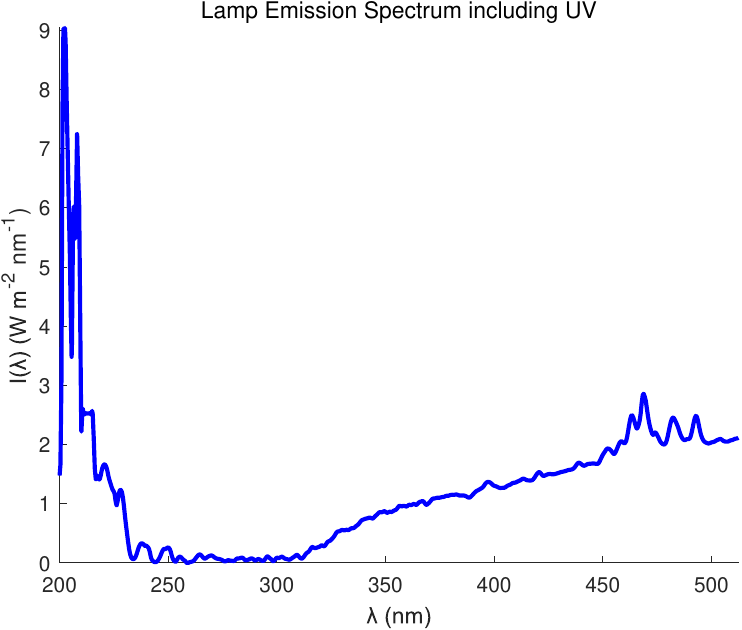}
    \caption{\hyperref[defn: UVT]{UVT}: Absolute irradiance as a function of the wavelength for the xenon lamp, including UV radiation.}
    \label{fig:Irrad_UV}
  \end{minipage}\hspace{0.05\textwidth} 
  \begin{minipage}{0.45\textwidth}
    \centering
    \includegraphics[width=.9\linewidth]{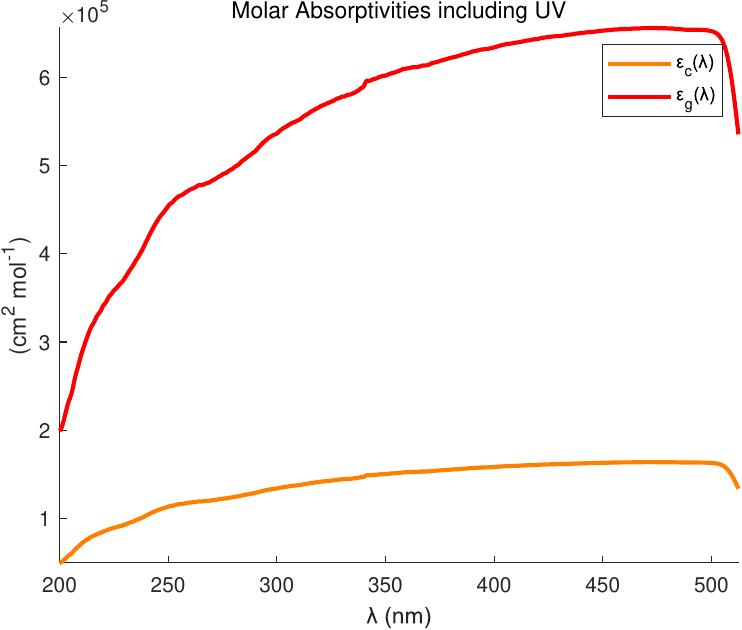}
    \caption{\hyperref[defn: UVT]{UVT}: Molar absorptivities as functions of the wavelength, including UV radiation.}
    \label{fig:Eps_UV}
  \end{minipage}
\end{figure}

The outcomes of the PC numerical simulations of the \hyperref[defn: UVT]{UVT} are shown in Figures \ref{fig:UVT_2D} and \ref{fig:UVT}. It is evident that, compared to the \hyperref[defn: BCT]{BCT}, the reaction is slower at the early stages but then the degradation becomes more pronounced and penetrates deeper into the painting (see also Figures  \ref{fig:UVT_2D_Zoom} and \ref{fig:UVT_Zoom}). To quantify this effect, we introduce the cumulative concentration of CdS at time $t,$ defined as follows
\begin{equation}\label{eq:C(t)}
    C_L^n= \Delta z  \sum_{j=0}^{N_z}c_j^n \approx \int_0^L c(\zeta,t_n) \ d \zeta, \qquad \qquad n=0,\dots, N_t,
\end{equation}
whose evolution is represented in Figures \ref{fig:UVT_Changes} and \ref{fig:UVT_Changes_All}. Therefore, considering the impact of UV light, we observe a final increase of $3.5\%$ in the overall deterioration process, compared to the base case. These findings are consistent with real-world observations, as increased deterioration is documented in artworks exposed to outdoor environments with natural light.

\begin{figure}[htbp]
  \centering
  \begin{minipage}{0.45\textwidth}
    \centering
    \includegraphics[width=.9\linewidth]{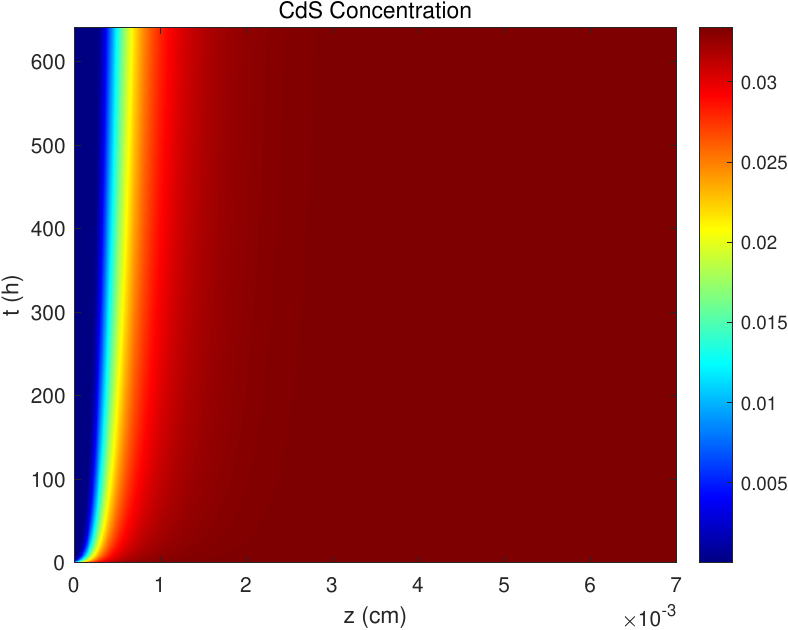}
    \caption{\hyperref[defn: UVT]{UVT}: Space-time evolution of the cadmium sulfide concentration}
    \label{fig:UVT_2D}
  \end{minipage} \hspace{0.05\textwidth} 
  \begin{minipage}{0.45\textwidth}
    \centering
    \includegraphics[width=.9\linewidth]{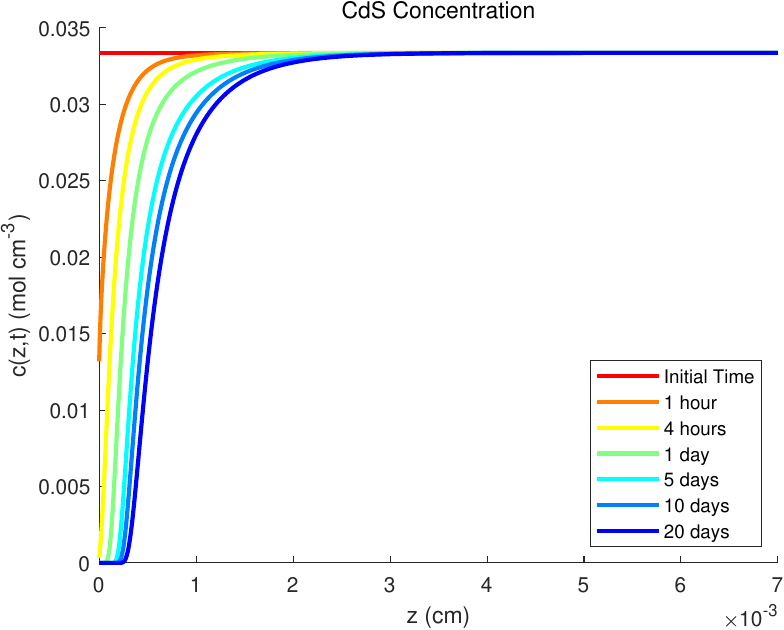}
    \caption{\hyperref[defn: UVT]{UVT}: Space evolution of the cadmium sulfide concentration at different times}
    \label{fig:UVT}
  \end{minipage}
\end{figure}

\begin{figure}[htbp]
  \centering
  \begin{minipage}{0.45\textwidth}
    \centering
    \includegraphics[width=.9\linewidth]{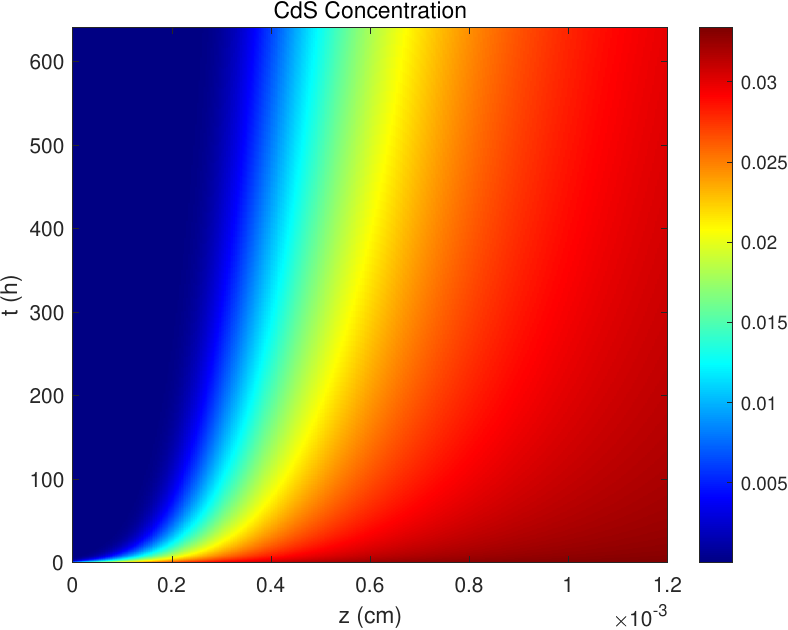}
    \caption{\hyperref[defn: UVT]{UVT}: Space-time evolution, near the surface, of the cadmium sulfide concentration.}
    \label{fig:UVT_2D_Zoom}
  \end{minipage} \hspace{0.05\textwidth} 
  \begin{minipage}{0.45\textwidth}
    \centering
    \includegraphics[width=.9\linewidth]{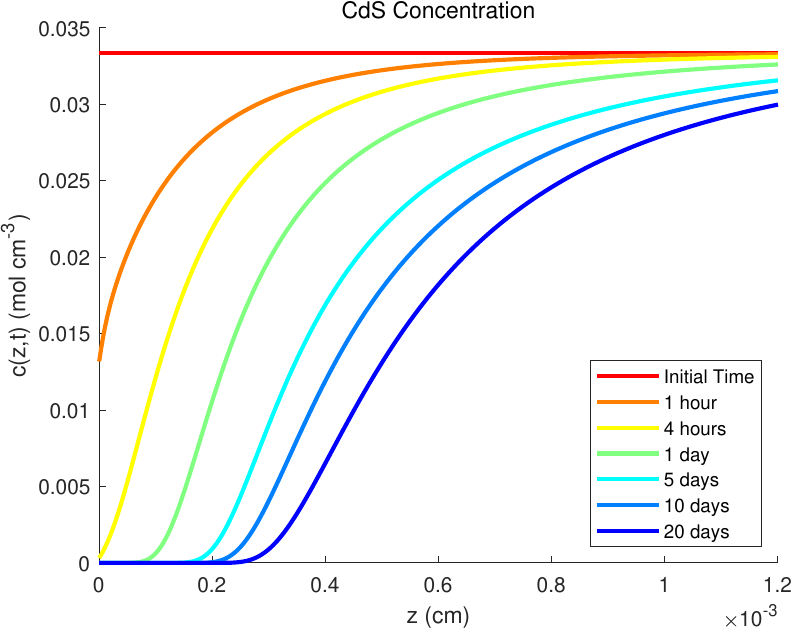}
    \caption{\hyperref[defn: UVT]{UVT}: Space evolution, near the surface, of the cadmium sulfide concentration for different times.}
    \label{fig:UVT_Zoom}
  \end{minipage}
\end{figure}
\begin{figure}[htbp]
  \centering
  \begin{minipage}{0.45\textwidth}
    \centering
    \includegraphics[width=.9\linewidth]{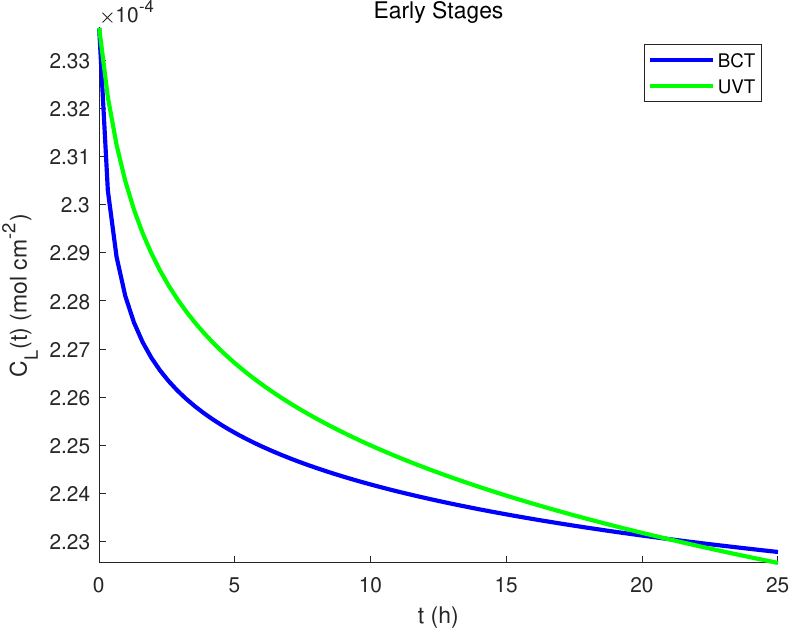}
    \caption{\hyperref[defn: UVT]{UVT}: Initial evolution of the cumulative cadmium sulfide concentration.}
    \label{fig:UVT_Changes}
  \end{minipage} \hspace{0.05\textwidth} 
  \begin{minipage}{0.45\textwidth}
    \centering
    \includegraphics[width=.9\linewidth]{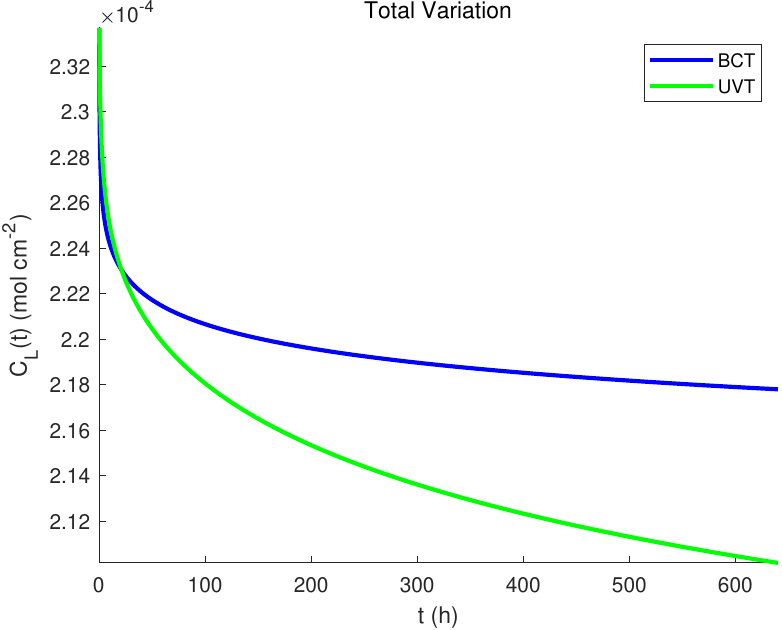}
    \caption{\hyperref[defn: UVT]{UVT}: Evolution of the cumulative cadmium sulfide concentration.}
    \label{fig:UVT_Changes_All}
  \end{minipage}
\end{figure}

\vspace{0.5\baselineskip}
\textbf{Test 4: the WPT}\label{defn: WPT} - We conduct a \textit{Water Penetration Test} (WPT) in order to assess the impact of humidity on the entire process. For the \hyperref[defn: BCT]{BCT} we assumed the water profile detailed in \eqref{eq:acqua_livello}, which decreases with increasing depth. Here, given the values of $w_b$ and $w^{ref}$ reported in Table \ref{tab:Parameters_Base_Case}, we consider the choices
\begin{equation*}
    w_s(z)=\chi_{[0,L_s]}(z) \left(1-\dfrac{w_b}{w^{ref}}\right)(1-z)    \qquad  \text{and} \qquad w_c(z)= \left(1-\dfrac{w_b}{w^{ref}}\right),
\end{equation*}
where $\chi_{[0,L_s]}$ represents the indicator function of the interval $[0,L_s].$ From a physical point of view, the function $w_s(z)$ models a heterogeneous medium that is permeable only within a surface sheet of thickness $L_s=3 \mu \text{m}$, while it is completely impermeable elsewhere. Conversely, the constant function $w_c(z)$ represents a scenario in which the water content remains uniform throughout the entire painting.

The numerical solution presented in Figure \ref{fig:WPT_barrier_Diff} demonstrates that, when $w(z)=w_s(z),$  the reaction is strictly confined to regions where the water penetrates and no degradation occurs for $z>L_s.$ Furthermore, the plot in Figure \ref{fig:WPT_barrier_Diff}, which shows the difference with the numerical solution of the \hyperref[defn: BCT]{BCT}, confirms that for $z\leq L_s$ the reaction behaves exactly in the same way as in the base case. 

The PC simulations, when $w(z)=w_c(z),$ result in the plots of Figures  \ref{fig:WPT_Const} and \ref{fig:WPT_Const_Diff}. In this case, it is evident that the higher water level compared to the \hyperref[defn: BCT]{BCT} enhances the reaction rate, resulting in a lower $CdS$ concentration for the \hyperref[defn: WPT]{WPT}.

\begin{figure}[htbp]
  \centering
  \begin{minipage}{0.45\textwidth}
    \centering
    \includegraphics[width=.9\linewidth]{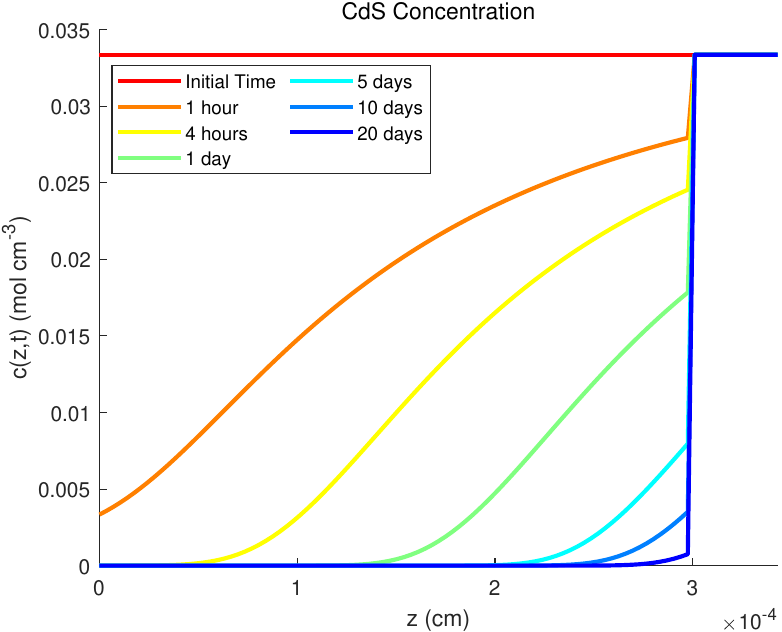}
    \caption{\hyperref[defn: WPT]{WPT}: Space evolution, near the superficial sheet, of the cadmium sulfide concentration. Here, $w(z)=w_s(z)$ and $L_s=3 \mu \text{m}$.}
    \label{fig:WPT_barrier}
  \end{minipage} \hspace{0.05\textwidth} 
  \begin{minipage}{0.45\textwidth}
    \centering
    \includegraphics[width=.9\linewidth]{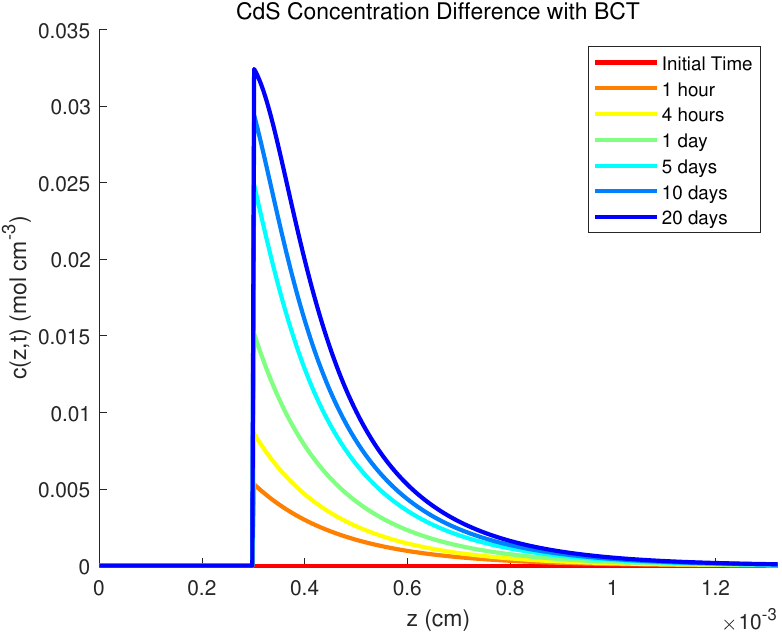}
    \caption{\hyperref[defn: WPT]{WPT}: Space evolution, near the surface, of the $CdS$ concentration difference with respect to the \hyperref[defn: BCT]{BCT}. Here, $w(z)=w_s(z)$ and $L_s=3 \mu \text{m}$.}
    \label{fig:WPT_barrier_Diff}
  \end{minipage}
\end{figure}
\begin{figure}[htbp]
  \centering
  \begin{minipage}{0.45\textwidth}
    \centering
    \includegraphics[width=.9\linewidth]{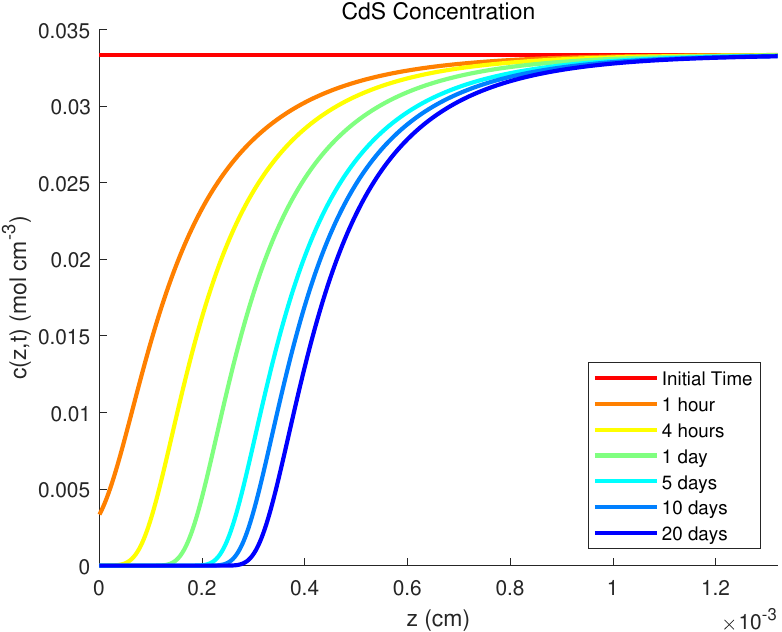}
    \caption{\hyperref[defn: WPT]{WPT}: Space evolution, near the surface, of the cadmium sulfide concentration. Here, $w(z)=w_c(z)$.}
    \label{fig:WPT_Const}
  \end{minipage} \hspace{0.05\textwidth} 
  \begin{minipage}{0.45\textwidth}
    \centering
    \includegraphics[width=.9\linewidth]{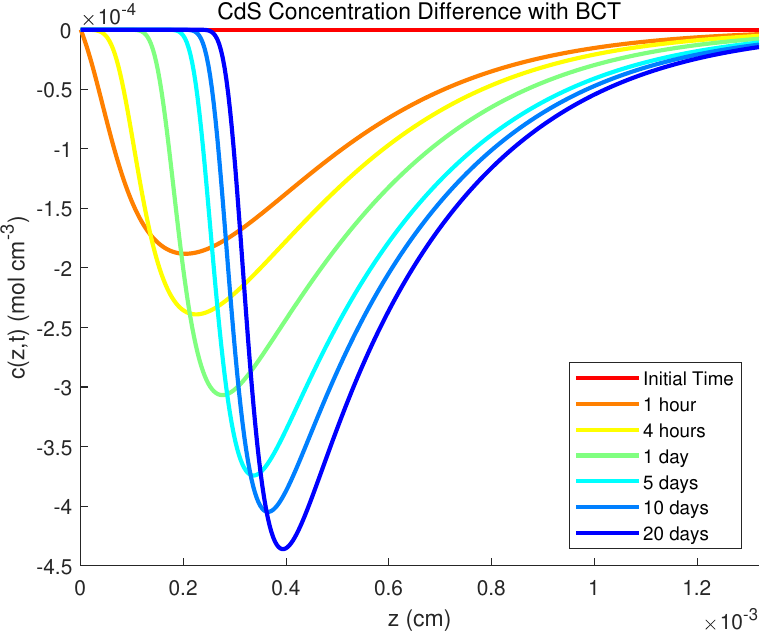}
    \caption{\hyperref[defn: WPT]{WPT}: Space evolution, near the surface, of the $CdS$ concentration with respect to the \hyperref[defn: BCT]{BCT}. Here, $w(z)=w_c(z)$.}
    \label{fig:WPT_Const_Diff}
  \end{minipage}
\end{figure}

\section{Sensitivity Analysis}\label{sec:Sensitivity}
Sensitivity Analysis (SA) assesses how variations in input parameters impact the model's output and identifies the most influential factors. Conducting the SA on a dimensionless model offers significant advantages, primarily due to the reduced number of input variables and the absence of dependencies on measurement units. In this section, we perform a comprehensive local SA of the model \eqref{eq:Continuous_model:ADIM}. Let $\nu_{0},\mu_{0}$ and $\xi_{0}$ denote the values of the dimensionless parameters corresponding to the \hyperref[defn: BCT]{BCT}, reported in Table \ref{tab:Parameters_Base_Case}. We define the relative overall $CdS$ concentration as follows
\begin{equation}\label{eq:Rel_overall_CdS}
    \mathcal{C}(\nu,\mu,\xi)=\dfrac{\Delta t \Delta z \sum_{n=0}^{N_t}\sum_{j=0}^{N_z}c_j^n(\nu,\mu,\xi)}{\Delta t \Delta z \sum_{n=0}^{N_t}\sum_{j=0}^{N_z}c_j^n(\nu_0,\mu_0,\xi_0)}\approx\dfrac{\int_0^1\int_0^1 c(z,t;\nu,\mu,\xi) \ dz \ dt}{\int_0^1\int_0^1 c(z,t;\nu_{0},\mu_{0},\xi_{0}) \ dz \ dt},
\end{equation}
where $c_j^n(\nu,\mu,\xi)$ represents the PC numerical solution to \eqref{eq:Continuous_model:ADIM} for the chosen values of the inputs $\nu,$ $\mu$ and $\xi$ ($0\leq n\leq N_t,$ $0\leq j \leq N_z$). 

For our analysis we consider various values of the input variables \((\nu, \mu, \xi)\) in a neighborhood of \((\nu_0, \mu_0, \xi_0)\) and investigate the resulting variations, compared to the base case, of the overall $CdS$ concentration. Specifically, we account for increments and decrements of $1\%$ up to a maximum of $10\%$ for each parameter. For all possible combinations, we apply the numerical method \eqref{eq:PC method ADIM} with stepsizes $\Delta z = \Delta t = \Delta \lambda = 5 \cdot 10^{-3}$ and compute the corresponding values of the function in \eqref{eq:Rel_overall_CdS}. The described procedure, involving $8000$ simulations executed concurrently in a parallel computing environment, yields the plots presented in Figures \ref{fig:3D_C_All} and \ref{fig:2D_C}, which illustrate the joint interplay between the input variables.
\begin{figure}[htbp]
  \centering
  \label{fig:3D_C_All}\includegraphics[width=0.90\linewidth]{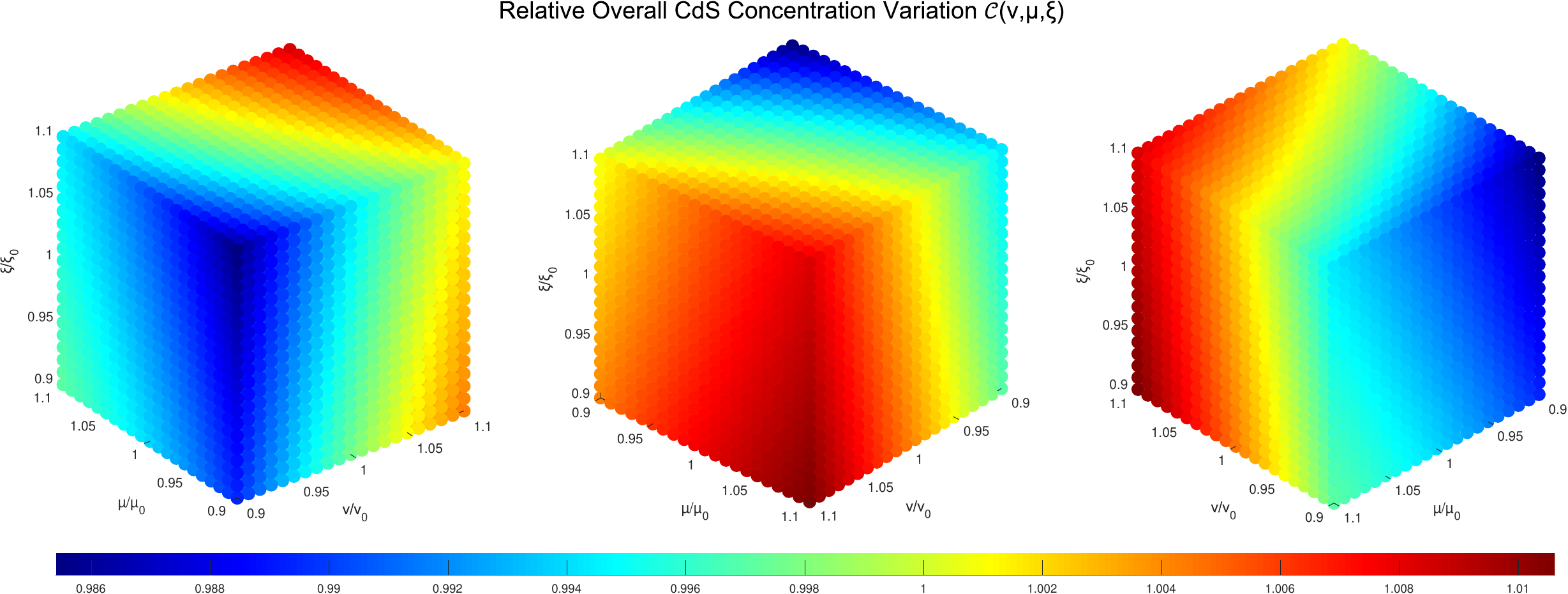}
  \caption{Sensitivity analysis of the dimensionless model \eqref{eq:Continuous_model:ADIM}: relative overall $CdS$ concentration as a function of the parameters variation.}
\end{figure}
\begin{figure}[htbp]
  \centering
  \label{fig:2D_C}\includegraphics[scale=1]{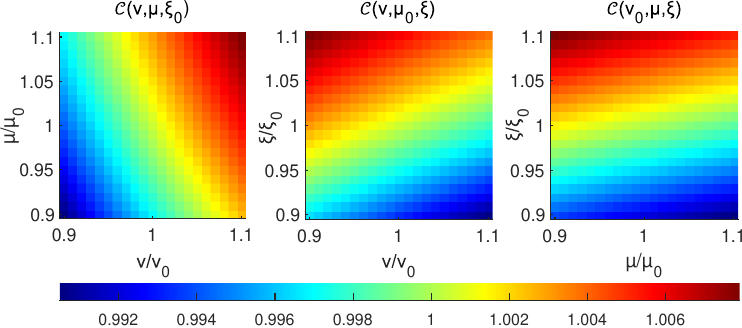}
  \caption{Sensitivity analysis of the dimensionless model \eqref{eq:Continuous_model:ADIM}: relative overall $CdS$ concentration as a function of the parameters variation on the planes $\xi=\xi_0$ (left panel), $\mu=\mu_0$ (central panel) and $\nu=\nu_0$ (right panel).}
\end{figure}

In order to provide quantitative insights on the sensitivity, we adopt the One-factor-At-a-Time (OAT) SA approach (see, for instance, \cite[Section 2.4.2]{Saltelli_2007}), which consists in altering only a single input parameter between multiple simulations. The outcomes of the OAT experiments, presented in Table \ref{tab:Sesitivity}, indicate that  $\mu$ and $\xi$ exert the greatest and the lowest impact on the relative overall $CdS$ concentration, respectively. Following the methodology outlined in \cite{Ciavolella2024,Sensitivity_Natalini}, we measure the sensitivity via the approximated partial derivatives
\begin{equation*}
    \begin{split}
        & \dfrac{\partial \mathcal{C}}{\partial \nu}(\nu_0 \pm \delta^k_\nu ,\mu_0,\xi_0)\approx\dfrac{\mathcal{C}(\nu_0 \pm \delta^{k+1}_\nu,\mu_0,\xi_0)-\mathcal{C}(\nu_0 \pm \delta^k_\nu,\mu_0,\xi_0)}{\delta^{k+1}_\nu-\delta^{k}_\nu}, \\
        & \dfrac{\partial \mathcal{C}}{\partial \mu}(\nu_0,\mu_0\pm \delta_\mu^{k},\xi_0)\approx\dfrac{\mathcal{C}(\nu_0,\mu_0\pm \delta_\mu^{k+1},\xi_0)-\mathcal{C}(\nu_0,\mu_0\pm \delta_\mu^{k},\xi_0)}{\delta^{k+1}_\mu-\delta^{k}_\mu}, \\
        & \dfrac{\partial \mathcal{C}}{\partial \xi}(\nu_0,\mu_0,\xi_0\pm \delta_\xi^{k})\approx\dfrac{\mathcal{C}(\nu_0,\mu_0,\xi_0\pm \delta_\xi^{k+1})-\mathcal{C}(\nu_0,\mu_0,\xi_0\pm \delta_\xi^{k})}{\delta^{k+1}_\xi-\delta^{k}_\xi},
    \end{split}
\end{equation*}
where $\delta^k_\theta=\frac{k \theta_0}{100}$ for $\theta\in\{\nu,\mu,\xi\}$ and $k=0,\dots,10.$ The numerical simulations in Figure \ref{fig:1D_C} highlight the monotonic behavior of the function $\mathcal{C}$, showing that it increases as $\nu$ or $\mu$ increases. Conversely, higher values of the input $\xi$ correspond to lower values of $\mathcal{C}$. 

\begin{figure}[htbp]
  \centering
  \label{fig:1D_C}\includegraphics[scale=1]{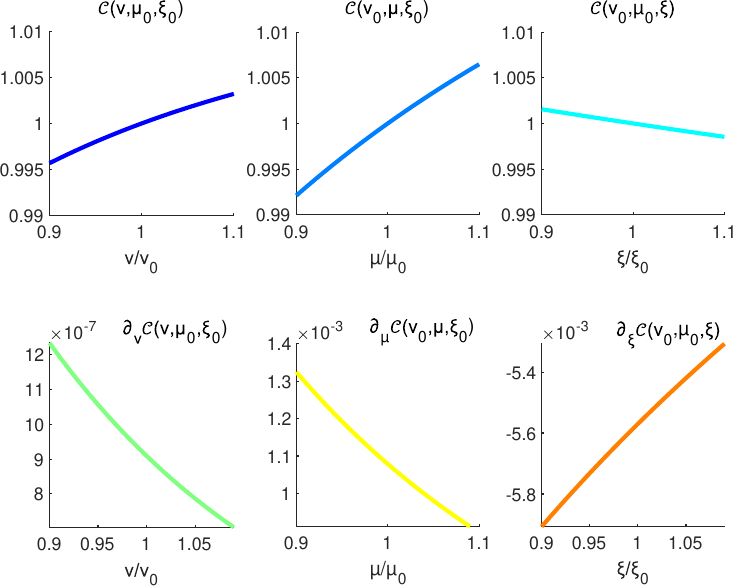}
  \caption{Sensitivity analysis via the OAT approach of the dimensionless model \eqref{eq:Continuous_model:ADIM}: relative overall $CdS$ concentrations and corresponding derivatives.}
\end{figure}

\begin{table}[htbp]
\footnotesize
\caption{Sensitivity analysis via the OAT approach of the dimensionless model \eqref{eq:Continuous_model:ADIM}: for $\nu=(1 \pm \delta_\nu) \nu_0,$ $\mu=(1\pm \delta_\mu) \mu_0$ and $\xi=(1 \pm \delta_\xi) \xi_0,$ the values of $\mathcal{C}$ are compared to $\mathcal{C}(\nu_0, \mu_0, \xi_0)=1.$ }\label{tab:Sesitivity}
\begin{center}
    \begin{tabular}{c|c||c|c||c|c}
        \textbf{Par. Var.} & \textbf{Function}  $\bm{\mathcal{C}}$ & \textbf{Par. Var.} & \textbf{Function}  $\bm{\mathcal{C}}$ & \textbf{Par. Var.} & \textbf{Function}  $\bm{\mathcal{C}}$ \\
        $\bm{\delta_\nu}$ & \textbf{Variation}   & $\bm{\delta_\mu}$ & \textbf{Variation} & $\bm{\delta_\xi}$  & \textbf{Variation}   \\
        \hline \rule{-3pt}{9pt}
         $+10 \%$ & $+0.321 \%$ & $+10 \%$ & $+0.648 \%$ & $+10 \%$ & $-0.147 \%$\\
    $+9 \%$ & $+0.292 \%$ & $+9 \%$ & $+0.588 \%$ & $+9 \%$ & $-0.132 \%$ \\
    $+8 \%$ & $+0.264 \%$ & $+8 \%$ & $+0.528 \%$ & $+8 \%$ & $-0.118 \%$\\
    $+7 \%$ & $+0.234 \%$ & $+7 \%$ & $+0.466 \%$ & $+7 \%$ & $-0.104 \%$\\
    $+6 \%$ & $+0.203 \%$ & $+6 \%$ & $+0.403 \%$ & $+6 \%$ & $-0.089 \%$ \\
    $+5 \%$ & $+0.172 \%$ & $+5 \%$ & $+0.339 \%$ & $+5 \%$ & $-0.074 \%$ \\
    $+4 \%$ & $+0.139 \%$ & $+4 \%$ & $+0.274 \%$ & $+4 \%$ & $-0.060 \%$ \\
    $+3 \%$ & $+0.106 \%$ & $+3 \%$ & $+0.207 \%$ & $+3 \%$ & $-0.045 \%$ \\
    $+2 \%$ & $+0.072 \%$ & $+2 \%$ & $+0.140 \%$ & $+2 \%$ & $-0.030 \%$ \\
    $+1 \%$ & $+0.036 \%$ & $+1 \%$ & $+0.070 \%$ & $+1 \%$ & $-0.015 \%$ \\
    $-1 \%$ & $-0.038 \%$ & $-1 \%$ & $-0.072 \%$ & $-1 \%$ & $+0.015 \%$ \\
    $-2 \%$ & $-0.076 \%$ & $-2 \%$ & $-0.145 \%$ & $-2 \%$ & $+0.030 \%$ \\
    $-3 \%$ & $-0.116 \%$ & $-3 \%$ & $-0.220 \%$ & $-3 \%$ & $+0.046 \%$ \\
    $-4 \%$ & $-0.157 \%$ & $-4 \%$ & $-0.296 \%$ & $-4 \%$ & $+0.061 \%$ \\
    $-5 \%$ & $-0.199 \%$ & $-5 \%$ & $-0.374 \%$ & $-5 \%$ & $+0.077 \%$ \\
    $-6 \%$ & $-0.243 \%$ & $-6 \%$ & $-0.453 \%$ & $-6 \%$ & $+0.092 \%$ \\
    $-7 \%$ & $-0.288 \%$ & $-7 \%$ & $-0.535 \%$ & $-7 \%$ & $+0.108 \%$ \\
    $-8 \%$ & $-0.334 \%$ & $-8 \%$ & $-0.617 \%$ & $-8 \%$ & $+0.124 \%$ \\
    $-9 \%$ & $-0.382 \%$ & $-9 \%$ & $-0.702 \%$ & $-9 \%$ & $+0.139 \%$ \\
    $-10 \%$ & $-0.431 \%$ & $-10 \%$ & $-0.789 \%$ & $-10 \%$ & $+0.155 \%$ \\
    \hline
    \end{tabular} 
\end{center}
\end{table}

From our OAT analysis, we conclude that the model \eqref{eq:Continuous_model:ADIM} exhibits relatively low sensitivity to perturbations of the parameters. Specifically, small percentage variations in the inputs yield a change in the output \eqref{eq:Rel_overall_CdS} that is an order of magnitude smaller than the perturbations themselves. This property demonstrates the model's robustness and reliability for practical applications, indicating that minor inaccuracies in parameter estimation do not significantly affect the representation of the phenomenon.


\section{Conclusions and Future Perspectives}\label{sec:Conclusions}
In this manuscript, we formulated a comprehensive mathematical framework to describe and simulate the degradation of cadmium yellow in pictorial matrices. Firstly, we proposed a model based on non-linear integro-differential equations that integrates Arrhenius and Beer-Lambert laws to track the photochemical conversion of cadmium sulfide to cadmium sulfate. We then developed a quadratically convergent, positivity-preserving numerical method built upon a predictor-corrector discretization of an exponential reformulation of the equations. Furthermore, we provided several simulations which demonstrate the effectiveness of the model in capturing some key aspects of the phenomenon, such as the passivation effect due to $CdSO_4$ accumulation on the surface and the increased degradation resulting from ultraviolet light exposure.  These outcomes, along with the sensitivity analysis we performed, confirm the robustness of the proposed framework as a reliable tool for assessing cadmium pigments deterioration.

Given the promising results of our work, there are some potential extensions we plan to investigate in future research. In this context, the implementation of a more involved light-penetration mechanism based on the Kubelka-Munk theory \cite{ciani2005determination,ciani2005light,ciani2005photodegradation} could further enhance the model's accuracy and applicability. The introduction of equations that account for the oxidation and yellowing of the binder used with the pigments may reveal of high interest, as well. Moreover, a calibration of the model's parameters with experimental data could provide valuable insights into the relationship between changes in $CdS$ concentration and observable discoloration in artworks.

\section*{Acknowledgments}
Maurizio Ceseri passed away before this manuscript was finalized. We have attempted to present the results of our collaboration in accordance with his high standards. This paper is dedicated to him, whose untimely death is a great loss to us all. \\
This work has been performed under the Project PE 0000020 CHANGES - CUP  B53C22003780006, NRP Mission 4 Component 2 Investment 1.3, Funded by the European Union - NextGenerationEU and in the auspices of the \textit{National Group for Mathematical Analysis, Probability and their Applications} (GNAMPA) and of the \textit{Italian National Group for Scientific Computing} (GNCS) of the National Institute for Advanced Mathematics (INdAM). The work of MP is partially supported by the INdAM under the GNCS Project CUP E53C23001670001.




\end{document}